\documentclass[final]{siamart1116}

\usepackage{enumerate}
\usepackage{comment}
\usepackage{xcolor}
\usepackage{subfigure}
\usepackage{dsfont}
\usepackage{lipsum}
\usepackage{amsfonts}
\usepackage{graphicx}
\usepackage{epstopdf}
\usepackage{algorithmic}
\ifpdf
  \DeclareGraphicsExtensions{.eps,.pdf,.png,.jpg}
\else
  \DeclareGraphicsExtensions{.eps}
\fi

\numberwithin{theorem}{section}

\newcommand{\TheTitle}{{\scriptsize Splitting schemes \& segregation in reaction-(cross-)diffusion systems}} 
\newcommand{\TheAuthors}{{\scriptsize J. A. Carrillo, S. Fagioli, F.  Santambrogio, and M. Schmidtchen}}

\headers{\TheTitle}{\TheAuthors}

\title{{\TheTitle}\thanks{JAC was supported by the Engineering and Physical Sciences Research Council (EPSRC) under grant no. EP/P031587/1. Most of the work was performed during a visit of FS to Imperial College thanks to a joint Imperial-CNRS Fellowship, the hospitality and support of both Imperial College and CNRS are warmly acknowledged. SF was partially supported by the local fund of the University
of L’Aquila "DP-LAND" (Deterministic Particles for Local And Nonlocal Dynamics), and by the
Erasmus Mundus programme "MathMods", www.mathmods.eu. SF thanks the Department of Mathematics at Imperial College in London for hosting him for a visiting period during which this work was carried out.}}

\author{
  J. A. Carrillo\thanks{Department of Mathematics, Imperial College London, SW7 2AZ London, UK.
    \email{carrillo@imperial.ac.uk}.}
  \and
  S. Fagioli\thanks{DISIM - Department of Information Engineering, Computer Science and Mathematics, University of L'Aquila, Via Vetoio 1 (Coppito) 67100 L'Aquila (AQ) - Italy
    \email{simone.fagioli@univaq.it}.}
  \and
  F.  Santambrogio\thanks{Laboratoire de Math\'ematiques d'Orsay, Univ. Paris-Sud, CNRS, Universit\'e Paris-Saclay, 91405 Orsay Cedex, FRANCE \email{filippo.santambrogio@math.u-psud.fr}.
  }
  \and M. Schmidtchen\thanks{Department of Mathematics, Imperial College London, SW7 2AZ London, UK. \email{m.schmidtchen15@imperial.ac.uk}.}
}

\usepackage{amsopn}

\ifpdf
\hypersetup{
  pdftitle={\TheTitle},
  pdfauthor={\TheAuthors}
}
\fi

\newcommand{\eE}{\mathcal{E}}

\newcommand{\R}{\mathbb{R}}

\DeclareMathOperator*{\argmin}{argmin}
\DeclareMathOperator*{\argmax}{argmax}

\newcommand{\supp}{\mathrm{supp}}
\newcommand{\sigmahalf}{\sigma^{n+1/2}}
\newcommand{\rhalf}{r^{n+1/2}}
\newcommand{\rhohalf}{\rho^{n+1/2}}
\newcommand{\etahalf}{\eta^{n+1/2}}
\newcommand{\Uhalf}{U^{n+1/2}}

\newcommand{\id}{\mathrm{id}}

\newcommand{\sigmanp}{\sigma^{n+1}}
\newcommand{\rnp}{r^{n+1}}
\newcommand{\rhonp}{\rho^{n+1}}
\newcommand{\etanp}{\eta^{n+1}}
\newcommand{\Unp}{U^{n+1}}

\newcommand{\sigman}{\sigma^{n}}
\newcommand{\rn}{r^{n}}
\newcommand{\rhon}{\rho^{n}}

\newcommand{\Un}{U^{n}}

\newcommand{\BV}[1]{\|#1\|_{BV}}
\newcommand{\Linfty}[1]{\|#1\|_{L^\infty}}
\newcommand{\N}{\mathbb{N}}
\newcommand{\ddt}{\frac{\mathrm{d}}{\mathrm{d} t}}
\renewcommand{\d}{\,\mathrm{d}}
\newcommand{\chii}{\chi'}

\newcommand{\Wt}{d_2^2}
\newcommand{\Wo}{d_1}
\newcommand{\Wp}{d_p^p}

\newcommand{\Bl}{d_{\mathrm{BL}}}

\newcommand{\curlyO}{\mathcal{O}}
\newcommand{\curlyK}{\mathcal{K}}
\newcommand{\curlyM}{\mathcal{M}}
\newcommand{\curlyP}{\mathcal{P}}
\newcommand{\curlyT}{\mathcal{T}}

\newtheorem{remark}[theorem]{Remark}

\begin{document}

\maketitle

\begin{abstract}
One of the most fascinating phenomena observed in reaction-diffusion systems is the emergence of segregated solutions, \emph{i.e.} population densities with disjoint supports. We analyse such a reaction cross-diffusion system. In order to prove existence of weak solutions for a wide class of initial data without restriction about their supports or their positivity, we propose a variational splitting scheme combining ODEs with methods from optimal transport. In addition, this approach allows us to prove conservation of segregation for initially segregated data even in the presence of vacuum.
\end{abstract}



\section{Introduction}
In this work we consider the following reaction-diffusion system for the evolution on an interval $x\in\Omega$ of two species with population densities $\rho,\eta\geq0$ :
\begin{equation}\label{eq:main}
\begin{cases}
\partial_t \rho = \partial_x\left(\rho\partial_x \chii(\rho+\eta)\right)+\rho F_1(\rho,\eta)+\eta G_1(\rho,\eta), \\
\partial_t \eta = \partial_x\left(\eta \partial_x\chii(\rho+\eta)\right)+\eta F_2(\rho,\eta)+\rho G_2(\rho,\eta),
  \end{cases}
\end{equation}
where $\chi : \left[0,\infty\right)\to \left[0,\infty\right)$ is a $C^1$ super-linear function modelling nonlinear diffusion and $F_i$ and $G_i$, $i=1,2$ model the reaction phenomena. Systems of this type appear naturally in mathematical biology. A fundamental biological phenomenon in interactions among different biological species is the \emph{inhibition} or \emph{activation} of growth whenever two populations occupy the same habitat. One species may promote or suppress the proliferation of the other species. In models involving cells or bacteria, the limited growth of different cell types can be attributed to volume or size constraints of the individual cells forming the different populations. The diffusive part in \eqref{eq:main} was originally introduced in the seminal papers \cite{GP84} and \cite{BGHP85,BGH87} and exhibits an intriguing phenomenon: segregated densities remain separated at all times. 

In fact, nonlinear diffusions are natural ways to include volume filling effects into mathematical biology models, see \cite{HPVolume,CCVolume} in the case of the classical Keller-Segel system. They help to avoid blow-up in these aggregation models in a biologically meaningful way and lead generically to asymptotic stabilisation. 
In the absence of reactions, the system leads to the nonlinear diffusion equation
\begin{equation}
	\label{eq:diff_sigma}
	\partial_t \sigma = \partial_x\left(\sigma\partial_x\chii(\sigma)\right) = \partial_x^2 \beta(\sigma).
\end{equation}
with $\sigma=\rho+\eta$, in which $\chii(\sigma)$ models the resistance to compression of the whole group of individuals $\sigma$. The natural assumption on \eqref{eq:diff_sigma} in order to be a diffusion equation is $\beta'(s)>0$ for $s>0$, possibly degenerating at $s=0$, or equivalently $\chi''(s)>0$ for $s>0$. The particular relevant case of $\chi(s)=s^2/2$ can be understood as the mean-field limit of interacting particles with very localised repulsion, see \cite{Ol,BCM,BV,Klar}. 

Related reaction-diffusion models to \eqref{eq:main} appear in tissue growth models where cell adhesion and volume effects are important factors determining cell sorting in heterogeneous cell populations, see \cite{MurTog, CHS17, BDFS, CCS17}, and zebrafish lateral line patterning \cite{VS}. They are also basic building bricks for a variety of cancer invasion models in the literature \cite{Chaplain.2005,  Preziosi.2003, Stinner.2015, Johnston.2010,BerDaPMim,ArmPaiShe1,ArmPaiShe2,ArmPaiShe3,GerCha,DomTruGerCha,SfaKolHelLuk} in which the coupling with other biologically meaningful modelling factors such as extracellular matrix, enzymatic activators and other substances are taken into account. These works usually involve drift terms due to long range attraction and/or repulsion between individuals leading to related mathematical difficulties with respect to \eqref{eq:main}, see for instance \cite{KM17}.  

The nonlinear diffusion equation \eqref{eq:diff_sigma} is well-studied, see \cite{Vaz}, and it can be understood as a gradient flow, see \cite{Otto,AGS,Vil_1,Snt_B}, of the energy functional
\begin{equation*}
 \eE(\sigma)=\int_\Omega \chi(\sigma) \d x,
\end{equation*}
in the metric space of probability measures endowed with a suitable topology induced by the euclidean transport distance denoted by $d_2$. The wellposedness of solutions to the nonlinear diffusion equation \eqref{eq:diff_sigma} was obtained in \cite{Otto} by means of the so called \emph{JKO-scheme}, \emph{cf.} \cite{JKO98}, which is a particular case of the minimising movement scheme by De Giorgi, see \cite{Amin} and the references therein. The idea of such a scheme is to recursively construct a sequence by solving a minimisation problem in a certain metric space $(X,d)$. Given some initial condition $\bar\sigma$ for Eq. \eqref{eq:diff_sigma} and a fixed time step $0< \tau < 1$ we set $\sigma_\tau^0=\bar{\sigma}$, and then recursively define
\begin{equation}\label{eq:JKO_intro}
  \sigma_\tau^{n+1}\in \argmin\limits_{\sigma\in X} \left\{\frac{1}{2\tau}d^2(\sigma,\sigma_\tau^{n})+\eE(\sigma)\right\},
\end{equation}
for $n\in\N$. 
The seminal work of R.~J.~McCann \cite{McC97} shows that $\eE(\sigma)$ is displacement convex or geodesically convex on the metric space of probability measures on the line endowed with $d_2$ as soon as $\beta$ is nondecreasing on $(0, + \infty)$, called the McCann's condition in short. We also refer to \cite[Chap. 9]{AGS}, \cite[p. 26]{CJMTU}, or \cite[Chap. 17]{Vil_2} for this classical notion, and \cite{BC} for related issues. 
Upon choosing a proper time interpolation $\sigma_\tau$, it can be proven that the sequence $\left\{\sigma_\tau\right\}_\tau$ converges to a weak solution of Eq. \eqref{eq:diff_sigma}, see \cite{Vil_1,AGS,Vil_2,Snt_B} and the references therein. 

In this work, we propose a variational splitting scheme in order to construct weak solutions to the reaction-diffusion system \eqref{eq:main}. More precisely, we solve in an inner time stepping the diffusive part of the system by the JKO scheme related to the nonlinear diffusion equation \eqref{eq:diff_sigma}, and then we transport both densities $\rho,\eta$ through the flow generated by the equation for the total population $\sigma$. 
Note that in this step the total and individual masses of the populations are unchanged in time. In the second inner time stepping of the splitting scheme, we solve the system of ODEs parameterised by the spatial variable $x\in\Omega$ leading to the final approximation of our new population densities after a time step. This variational splitting scheme will be written in details in Section 2. The splitting between reaction and diffusion steps is natural from the numerical analysis view point as it has already been used for variations of Keller-Segel models where the diffusion step is solved by the JKO scheme \cite{Carrillo.2008,CKL17} in the case of a single population density coupled with a system of reaction-diffusion equations. 

Our main result shows the convergence of the splitting variational scheme towards weak solutions of the system \eqref{eq:main}. The main mathematical difficulty here arises from the cross-diffusion term allowing for segregation fronts to form in the solutions. This phenomenon was proven in \cite{BGHP85} in the case of initial data with separated supports for the populations. More precisely, while \cite{GP84} constructs a source solution of the system without reactions similar to the well-known Barenblatt-Pattle profiles \cite{Vaz} for nonlinear diffusions, \cite{BGHP85} constructs a solution to the system without reactions by formulating it as a free boundary problem for a single effective equation, and by characterising the segregation front through this free boundary. This approach can only work in case the support of both populations are at a positive distance to each other initially.
Later \cite{BerDaPMim, BHIM12} combined both the nonlinear diffusion and the reaction to obtain a system similar to \eqref{eq:main} showing similar segregation phenomena by regularisation techniques. However, their approach heavily relies on the absence of vacuum as they assume that $\sigma_0$ is bounded below by a positive constant.

These remarkable results have severe consequences --  initially smooth solutions lose their regularity when both densities meet each other. In fact, they become discontinuous at the contact interface immediately. This phenomenon legitimises our functional space choice as bounded functions of bounded variation, see the precise notion of weak solution and assumptions on the initial data in the next section. 

In contrast to \cite{BGHP85,BerDaPMim, BHIM12}, we show the convergence of our variational splitting scheme for general initial data even in the presence of vacuum and for general nonlinearities.
Moreover, we recover their result in \cite{BGHP85,BerDaPMim} about segregation fronts by showing that initial data which are initially segregated remain segregated for all times. In fact, we are even able to drop their restrictive assumption of strict boundedness away from zero, \emph{i.e.} absence of vacuum. An important technical point in our proof relies on displacement convexity of an auxiliary functional that allows to obtain further regularity on the approximate solutions in order to pass to the limit in the nonlinear diffusion terms. This auxiliary functional imposes a slightly more restricted set of nonlinear diffusions satisfying some integrability condition at the origin, see the precise conditions in the next section.

The rest of the paper is organised as follows: In Section \ref{sec:main_result} we introduce the variational splitting scheme, present the main result, and explain the strategy of the proof. Section \ref{sec:proof_main_result} is dedicated to deriving all estimates necessary for proving the existence theorem as well as the segregation theorem, and finally, in Section \ref{sec:numerics} we conclude by illustrating the result with some numerical examples.


\section{\label{sec:main_result}Preliminaries and main result}
As already mentioned, our main aim is to study the existence of weak solutions for the following  one-dimensional two species cross-diffusion and reaction system:
\begin{equation}\label{eq:complete}
\begin{cases}
\partial_t \rho = \partial_x\left(\rho\partial_x \chii(\rho+\eta)\right)+\rho F_1(\rho,\eta)+\eta G_1(\rho,\eta) ,&\quad \mbox{in } \left[0,T\right]\times\Omega, \\
  \partial_t \eta = \partial_x\left(\eta \partial_x\chii(\rho+\eta)\right)+\eta F_2(\rho,\eta)+\rho G_2(\rho,\eta), &\quad \mbox{in } \left[0,T\right]\times\Omega,\\
  \partial_x\chi'\left(\rho+\eta\right)(x,t)=0, & \quad \mbox{on } \left[0,T\right]\times\partial\Omega,\\
  \rho(\cdot,0)=\rho_0,\quad \eta(\cdot,0)=\eta_0, & \quad \mbox{in } \Omega,
\end{cases}
\end{equation}
where $\Omega\subset \R$ is an open bounded interval, $T>0$.  Moreover,  $\chi$ denotes an internal energy density and $F_i$ and $G_i$, $i=1,2$ model the reaction phenomena. As mentioned before the space of bounded functions with bounded variation is a natural functional setting. 

\begin{definition}[\label{def:BV_space}Space of functions of bounded variations]
	Let $f:\bar \Omega \rightarrow \R$. We define its variation with respect to a partition $P:=\{x_1 < x_2 < \cdots < x_{|P|}\}\subset \Omega$ by
	\begin{align*}
		V_P(f):=\sum\limits_{i=1}^{|P|-1}|f(x_{i+1})-f(x_i)|.
	\end{align*}
	We call $f$ a function of bounded variation if its total variation $\sup_P V_P(f)<\infty$ is finite. Here the supremum is taken over all partitions of $\Omega$.
	We denote by $BV(\Omega)$ the set of functions whose variation is bounded. Equipped with the norm
	\begin{align*}
		\BV{f}:= \sup\limits_{P} V_P(f),
	\end{align*}
	the set $BV(\Omega)$ is a vector space.
\end{definition}

We shall see in the remainder of this section that the vector space of bounded function with bounded variation is a good choice to construct solutions. In our analysis we will exploit the following property.
\begin{lemma}[\label{lem:Ralgebra}$BV(\Omega)\cap L^\infty(\Omega)$ is an $\R$-algebra.]
	The vector space $BV(\Omega)\cap L^\infty(\Omega)$ equipped with the pointwise multiplication
	\begin{align*}
		\big(BV(\Omega)\cap L^\infty(\Omega)\big)^2\ni(f,g)\mapsto f g\in BV(\Omega)\cap L^\infty(\Omega),
	\end{align*}
	is a real algebra.
\end{lemma}

The proof of the previous result is standard. Notice that in one dimension, $BV$-regularity implies boundedness. We prefer to write $BV(\Omega)\cap L^\infty(\Omega)$ for the sake of clarity and possible future generalisations.

\subsection{\label{sec:metric_structure}Metric structure}
Consider $\Omega\subseteq \R$ an open bounded interval, and denote by $\curlyM_+(\Omega)$ the set of positive and finite measures. Throughout this paper we will make use of the following notation
\begin{align*}
	\curlyP^m(\Omega):=\left\{\mu \in \curlyM_+(\Omega)\, \big|\, \mu(\Omega)= m\right\},
\end{align*}
that is, the set of positive measures with mass $m>0$.  Consider a measure $\mu\in\curlyP^m(\Omega)$ and a Borel map $\curlyT:\R\rightarrow\R$. We denote by  $\nu = \curlyT_{\#}\mu\in\curlyP^m(\Omega)$ the push-forward measure of $\mu$ through $\curlyT$, defined by
\begin{equation*}
    \int_{\Omega}f(y)\d \curlyT_{\#}\mu(y)=\int_{\Omega}f(\curlyT(x))\d\mu(x),
\end{equation*}
for all Borel functions $f$ on $\Omega$. We call $\curlyT$ a \emph{transport map} pushing $\mu$ to $\nu$. We endow the space $\curlyP^m(\Omega)$ with the $p-$Wasserstein distance, $p\geq 1$, 
\begin{equation*}
     \Wp(\mu_1,\mu_2)=\inf_{\gamma\in \Pi^m(\mu_1,\mu_2)}\left\{ \int_{\Omega\times\Omega}|x-y|^{p}\d \gamma(x,y)\right\}.
\end{equation*}
Here, $\Pi^m(\mu_1,\mu_2)$ is the set of all transport plans between $\mu_1$ and $\mu_2$, that is the set of positive measures of fixed mass, $\gamma \in\curlyP^m(\Omega\times\Omega)$, defined on the product space such that $\pi^{i}_{\#}\gamma=\mu_{i}$, for $i=1,2$,  where $\pi^{i}$ denotes the projection operator on the $i$-th component of the product space. If $\mu_1$ is absolutely continuous with respect to the Lebesgue measure the optimal transport map, $\gamma$, is unique and can be written as $\gamma = (\mathrm{id}, \curlyT)_\#\mu$. In addition there exists a \emph{Kantorovich potential}, $\varphi$, that is linked to the transport map, $\curlyT$, in the following way
\begin{align}
	\label{eq:rel_T_KantorovichPotential}
	\curlyT(x) = (\id - \partial_x \varphi)(x).
\end{align}
We refer to \cite{Vil_1,AGS,Vil_2,Snt_B} and the references therein for a good account of the properties of transport distances and the state of the art in gradient flows/steepest descents of functionals in metric spaces of probability measures. While transport distances are an incredibly powerful tool for dealing with transport PDEs exhibiting a gradient flow structure, it is not applicable in the presence of source terms. This is owing to the fact that it is only defined  for two measures of the same mass. 

To resolve this shortcoming we will make use of the Bounded-Lipschitz distance $\Bl$, classically used for the derivation of the Vlasov equation, see \cite{Neun,Spo,CCR,CCC,GM} and the references therein. The Bounded-Lipschitz distance $\Bl$, also frequently called \emph{flat metric}, is defined as follows
\begin{align*}
\Bl(\mu,\nu):=\sup \left\{\int_\Omega f \d\left(\mu-\nu\right)|\, \,\|f\|_{L^{\infty}(\Omega)},\|f'\|_{L^{\infty}(\Omega)}\leq 1\right\} = \|\mu-\nu\|_{({W^{1,\infty}(\Omega)})^\ast}.
\end{align*}
Since our problem is posed on the product space we extend the metric setting to the space $\mathcal{M}_+\times\mathcal{M}_+$, the product space of nonnegative measures in the canonical way. For $d\in\{\Bl, d_p\}$, $1\leq p<\infty$, we define the product metric (still denoted $d$) as
\begin{align}
	\label{eq:BL_dist}
	d(U,\tilde{U}):=d(\rho,\tilde{\rho})+d(\eta,\tilde{\eta})\,,
\end{align}
where $U=(\rho,\eta),\tilde{U}=(\tilde{\rho},\tilde{\eta})\in\curlyM_+\times\curlyM_+$.

\begin{proposition}[Properties of $\Bl$]
Let $\mu, \nu \in L_+^1(\Omega)$ be two densities. Then the following properties hold true:
\begin{enumerate}[(i.)]
\item $\Bl(\mu,\nu)\leq \|\mu-\nu\|_{L^1(\Omega)},$
\item $\Bl(\mu,\nu) \leq \Wo(\mu,\nu)$, whenever $\mu(\Omega) = \nu(\Omega)$.
\end{enumerate}
\end{proposition}
\begin{proof}
Given $\mu,\nu\in L_+^1(\Omega)$ be arbitrary and $f \in W^{1,\infty}(\Omega)$ with $\|f\|_{W^{1,\infty}(\Omega)} \leq 1$. Hence we may write
\begin{align*}
	\left|\int_\Omega f \d (\mu-\nu)\right| \leq \int_\Omega |f|\d |\mu-\nu| =\|\mu-\nu\|_{L^1(\Omega)}.
\end{align*}
Taking the supremum over all such functions $f$ we get the first statement by using definition \eqref{eq:BL_dist}. For the second statement, we additionally assume that $\mu(\Omega) = \nu(\Omega)$. We recall the dual definition of $\Wo$,
\[
\Wo(\mu,\nu)=\sup\left\{\int_\Omega f \d (\mu-\nu)\,|\,f\in \mathrm{Lip}(\Omega)\mbox{ s.t. } \|f\|_{\mathrm{Lip}(\Omega)} \leq  1\right\},
\]
where $\mathrm{Lip}(\Omega)$ is the set of Lipschitz-continuous functions, see \cite{Vil_1,Snt_B}. Then property (ii.) is a consequence of this formulation, and this concludes the proof.
\end{proof}
\begin{corollary}\label{cor:triangular:_inequality}
	Let $U_i=(\mu_i,\nu_i)\in \mathcal{M}_+(\Omega)^2$ for $i=1,2,3$. Furthermore assume $\mu_i,\nu_i\in L^1(\Omega)$ for $i=1,2$ as well as $\mu_2(\Omega)=\mu_3(\Omega)$ and $\nu_2(\Omega)=\nu_3(\Omega)$.
	Then
	\begin{align*}
		\Bl\big(U_1, U_3\big) \leq \|U_1 -U_2\|_{L^1(\Omega)} + \Wo\big(U_2,U_3\big).
	\end{align*}
\end{corollary}

Throughout two quantities are crucial for our analysis -- the sum of the two densities, $\sigma$, and the ratio, $r$, between one density, say $\rho$, and the sum 
\begin{equation*}
 \sigma = \rho + \eta, \qquad \text{and} \qquad r = \frac{\rho}{\sigma},
\end{equation*}
where we assume no vacuum, \textit{i.e.} $\sigma>0$. A straightforward computation shows these functions satisfy the following system of PDEs
\begin{align}\label{eq:transf_sys}
\left\{
\begin{array}{l}
	\partial_t \sigma = \partial_x\left(\sigma\partial_x\chii(\sigma)\right)+\sigma \left(r(\tilde F_1 +\tilde G_2) +(1-r)\left(\tilde G_1 + \tilde F_2\right)\right), \\[0.7em]
	\partial_t r =\partial_x r\partial_x \chii(\sigma) + r(1-r)\left(\tilde{F}_1-\tilde{F}_2\right)+(1-r)^2\tilde{G}_1-r^2\tilde{G}_2,
\end{array}
\right.
\end{align}
where we used $\tilde{F}_i(\sigma,r)=F_i(r\sigma,(1-r)\sigma)$ and $\tilde{G}_i(\sigma,r)=G_i(r\sigma,(1-r)\sigma)$, for $ i=1,2$, to denote the reaction terms in the transformed variables. In order to simplify the analysis in Section \ref{sec:proof_main_result}, let us introduce the more concise notation
\begin{align*}
	A_1(r,\sigma): = \tilde F_1 + \tilde G_2,\quad A_2(r,\sigma) := \tilde G_1 + \tilde F_2, \quad \mathrm{and} \quad  A_3(r,\sigma):= \tilde F_1 - \tilde F_2.
\end{align*}
Note that these functions are Lipschitz and bounded as they are linear combinations of $BV\cap L^\infty$ functions. Thus the transformed system \eqref{eq:transf_sys} can be rewritten in the more compact form
\begin{align}\label{eq:transf_sys1}
\left\{
\begin{array}{l}
	\partial_t \sigma = \partial_x\left(\sigma\partial_x\chii(\sigma)\right)+\sigma \big(r A_1 +(1-r)A_2\big),\\[0.5em]
	\partial_t r =\partial_x r\partial_x\chii(\sigma) + r(1-r) A_3 +(1-r)^2\tilde{G}_1-r^2\tilde{G}_2.
\end{array}
\right.
\end{align}
\noindent
Let us note here, that taking $F_i=G_i=0$ in \eqref{eq:main}, yields the following nonlinear cross-diffusion system
\begin{align}\label{eq:cross_diff}
\left\{
\begin{array}{l}
	\partial_t \rho = \partial_x\left(\rho\partial_x\chii(\sigma)\right),\\[0.5em]
	\partial_t \eta = \partial_x\left(\eta\partial_x\chii(\sigma)\right),
\end{array}
\right.
\end{align}
where the sum $\sigma$ satisfies the nonlinear diffusion equation \eqref{eq:diff_sigma}.

In order to be useful for our purposes, we need some properties on the internal energy density. The role of these properties will be clearer after the statement of our main result.
\begin{definition}[Internal energy density]
	\label{def:feasible_NLD}
	A function $\chi:[0,\infty] \rightarrow \R$ is an internal energy density if 
	\begin{enumerate}[({NL}-i)]
		\item $\chi \in C^0([0,\infty],\R) \cap C^2((0,\infty),\R)$ with $\chi''>0$,
		\item $\lim_{h\downarrow 0} \chi'(h) = 0$.
		\item the integrals $	\kappa(x) := \int_1^x \frac{\chi''(s)}{s}\d s$, and $K(\sigma):= \int_0^\sigma \kappa(x)\d x$ exist.
	\end{enumerate}
\end{definition}

As mentioned in the introduction, the space of probability measures endowed with $d_2$ has proven to be an exceptional choice of a metric space. In this case the minimiser in Eq. \eqref{eq:JKO_intro} satisfies the so-called optimality condition
\begin{align}
	\label{eq:optimality_condition}
	\frac{\varphi}{\tau} + \chi'(\sigma^{n+1}) = const.,
\end{align}
\emph{cf.} \cite[Proposition 7.20]{Snt_B}, for instance. Here $\varphi$ denotes the associated Kantorovich potential \cite[Theorem 1.17]{Snt_B}. Notice that this optimality is only obtained in those references for non-degenerate problems, where $\beta'(0^+)>0$. However, similar approximation techniques as those developed in \cite[p. 156]{Otto} and \cite[p. 27]{CJMTU} allow to overcome the difficulties associated to degenerate diffusions, where $\beta'(0^+)=0$.
In the rest of the paper, we will proceed as if we were dealing with nonlinearities leading to nondegenerate diffusions since by this standard approximation procedure, the same result can be obtained for the degenerate ones. We are now ready to introduce our notion of weak solutions. 

\begin{definition}[Notion of weak solutions]\label{def:weak_sol}
A couple $\rho,\eta\in  C\big( 0,T; BV(\Omega)\cap L^{\infty}(\Omega)\big)$ is a weak solution to system \eqref{eq:complete} if $\sigma=\rho+\eta\in L^2(0,T;H^1(\Omega))$  and there holds 
\begin{align*}
 \int_\Omega(\rho(t)-\rho(s))\,\zeta\d x &= \int_s^t\int_\Omega -\rho\partial_x\chi'(\rho+\eta)\partial_{x} \zeta + \left(\rho F_1(\rho,\eta)+\eta G_1(\rho,\eta)\right) \zeta \d x \d\bar\tau,\\
 \int_\Omega(\eta(t)-\eta(s))\,\xi\d x &= \int_s^t\int_\Omega -\eta\partial_x\chi'(\rho+\eta)\partial_{x} \xi + \left(\eta F_2(\rho,\eta)+\rho G_2(\rho,\eta)\right) \xi \d x \d\bar\tau,
\end{align*} 
for any two test functions $\zeta, \xi \in C_c^\infty(\Omega)$.
\end{definition}

\begin{remark}
\mbox{}
\\
\indent
\label{rem:K_geodesic_convex}
\textit{1.} Notice that the functional 
	\begin{align*}
		\curlyK(\sigma) := \int_\Omega K(\sigma)\d x,
	\end{align*}
associated to $\chi$, as in Definition \ref{def:feasible_NLD}, satisfies the McCann condition since $K''(s)=\tfrac{\chi''(s)}s>0$, and is therefore displacement convex. This fact will be crucial in Section \ref{sec:combi_estim} in order to prove Lemma \ref{lem:uniformH1_bound}. Let us just state here that it is necessary to obtain additional regularity from the dissipation of this functional on the nonlinear diffusion term, thus allowing us to pass to the limit in the approximating sequence.

\textit{2.}
	Let us note that we can also allow for nonlocal reaction terms, \emph{i.e.}
	\begin{align*}
		F_i = W_{1, i}\star \rho + W_{2,i}\star \eta,
	\end{align*}
	and similarly for $G_i$, for $i=1,2$. The only assumption we need to impose on the kernels is that they are smooth and integrable. These models are found in modelling  pattern formation, as for instance the kernel-based Turing pattern system \cite{Kon17} or the nonlinear aggregation-diffusion system \cite{VS}.

\textit{3.} Similarly, we can -- at least formally -- interpret the system \eqref{eq:cross_diff} as a gradient flow of the functional
\begin{equation*}
	\eE(\rho, \eta)=\int_\Omega \chi(\rho+\eta) \d x,
\end{equation*}	
in the product Wasserstein space.
\end{remark}

\subsection{Splitting scheme}
We are now ready to introduce our splitting scheme for equation \eqref{eq:complete}. Let some initial data $\rho_0,\eta_0\in BV(\Omega)\cap L^{\infty}(\Omega)$ be given such that there exists a function $r_0 \in BV(\Omega)$ such that
\begin{align*}
	\sigma_0:=\rho_0+\eta_0,\qquad\mbox{ and } \qquad  \frac{\rho_0}{\sigma_0}=r_0\big|_{\{\sigma_0>0\}},
\end{align*}
and $0\leq r_0 \leq1$. Furthermore we assume $F_i$ and $G_i$, $i=1,2$ are bounded and Lipschitz with respect to $\rho$ and $\eta$ and we impose $G_1(0,\cdot)\geq0$ and $G_2(\cdot,0)\geq0$ to ensure positivity of solutions.\\

We fix $0<\tau<1$ and $n\in \{1,\ldots, N\}$, with $N\in \N$ such that $N\tau=T$. We then recursively construct the piecewise constant approximation to the system as follows. We impose
\[
 (\rho_\tau^0,\eta_\tau^0)=(\rho_0,\eta_0),
\]
and then construct $\Unp=(\rhonp,\etanp)$ by the following scheme. We split the equation into a reaction step and a diffusion step on the time interval $[t^{n}, t^{n+1})$, with $t^n = n \tau$, for all $1\leq n\leq N$.

\subsubsection{Reaction step}
The reaction phase consists of solving the system of ordinary differential equations
\begin{align*}
\left\{
\begin{array}{l}
	\partial_t \sigma = \Sigma(\sigma, r):= \sigma \big(r A_1 +(1-r)A_2\big),\\[0.5em]
	\partial_t r = R(\sigma, r):=r(1-r) A_3 +(1-r)^2\tilde{G}_1-r^2\tilde{G}_2,\\[0.5em]
	\sigma(t^n)=\sigma^n, \quad\mbox{and}\quad r(t^n)=r^n,
\end{array}
\right.
\end{align*}
in the time interval $[t^n, t^{n+1})$.  We then set 
\begin{subequations}
\begin{align}
\label{eq:r_step}
\rhohalf := r\sigma\big|_{t= (n+1)\tau}, \quad \mbox{and}\quad \etahalf := (1-r)\sigma\big|_{t=(n+1)\tau}.
\end{align}
A straightforward computation reveals 
\begin{align}
	\label{eq:rsigma_eqn}
	\begin{split}
	\partial_t (r \sigma) &= r\sigma \tilde F_1 + (1-r)\sigma \tilde G_1,\\
	\partial_t ((1-r) \sigma) &= (1-r)\sigma \tilde F_2 + r\sigma \tilde G_2,
	\end{split}
\end{align}
verifying that $r\sigma$ and $(1-r)\sigma$ solve the reaction part of Eq. \eqref{eq:complete}.

\subsubsection{Diffusion step}
After the reaction phase we solve 
\begin{equation}\label{eq:JKO}
  \Unp\in \argmin\limits_{U\in\curlyP^{m_1}(\Omega)\times \curlyP^{m_2}(\Omega)}\left\{\frac{1}{2\tau}\Wt\left(U,\Uhalf\right)+\eE(U)\right\},
\end{equation}
where $m_1=\rhohalf(\Omega)$ and $m_2 = \etahalf(\Omega)$. Let $\curlyT$ denote the associated transport map pushing, \emph{i.e.} $\curlyT_\#\Unp=\Uhalf$. We define
\begin{align}
	\label{eq:JKOrsigma}
	\sigmanp:=\rhonp+\etanp, \qquad\mbox{ and }\qquad \rnp:=\rhalf\circ\curlyT.
\end{align}
\end{subequations}
While the definition of $\sigmanp$ is somewhat natural, the definition of $\rnp$ seems a bit surprising. Let us note here that, indeed, $\rnp = \rhonp/\sigmanp$ where $\sigmanp>0$. For the precise argument we refer to Section \ref{sec:estim_diffu}, Eq. \eqref{eq:composition_with_monotone_T}. 

\subsubsection{Combination of both steps -- construction of a solution}
Throughout this paper we refer to Eq. \eqref{eq:r_step} as \emph{reaction step} and to Eq. \eqref{eq:JKO} as \emph{diffusion step}, respectively. 
\begin{definition}\label{def:piec_con}[Piecewise constant interpolation]
	Let $(\rn, \sigman)_{n\in\N}$ be the sequen\-ce obtained from the splitting scheme. Then we define the piecewise constant interpolations by
	\begin{align*}
    	r_\tau(t,x) = \rn(x), \quad \text{and} \quad \sigma_\tau(t,x) = \sigman (x),
    \end{align*}
as well as
    \begin{align*}
    	\rho_\tau(t,x) = \rn(x)\sigman(x), \quad \text{and} \quad \eta_\tau(t,x) = (1-\rn (x))\sigman(x),
    \end{align*}
    for all $(x,t)\in \Omega \times [t^n,t^{n+1})$. Furthermore we write $U_\tau := (\rho_\tau,\eta_\tau)$.
\end{definition}

We will say that two densities $\rho,\eta\in BV(\Omega)\cap L^\infty(\Omega)$ are segregated if the intersection of the interior of their supports is empty. We are now ready to state our main result.

\begin{theorem}[Convergence to weak solutions] Let $\rho_0,\eta_0\in BV(\Omega)\cap L^\infty(\Omega)$ and assume there exists a function $r_0\in BV(\Omega)$ such that $r_0 = \rho_0/(\rho_0+\eta_0)$ on $\{\rho_0+\eta_0>0\}$ and $0\leq r_0\leq 1$. Then, upon the extraction of a subsequence, the piecewise constant approximations $(\rho_\tau)_{\tau>0}$ and $(\eta_\tau)_{\tau>0}$ converge to a weak solution of system \eqref{eq:complete} in the sense of Definition \ref{def:weak_sol}. Moreover, if initially the two densities $\rho_0,\eta_0$ are segregated, then the limit densities $\rho(t,\cdot),\eta(t,\cdot)$ remain segregated for all times.
\end{theorem}

\section{\label{sec:proof_main_result}Proof of the main result}
This section is dedicated to proving the main result of the paper -- the convergence of the approximation obtained by the splitting scheme to a solution of the system. It is organised as follows: in Section \ref{sec:estim_reac} and \ref{sec:estim_diffu} we establish the crucial BV-estimates and $L^\infty$-bounds. In Section \ref{sec:combi_estim} we combine the estimates from the previous sections in order to get uniform estimates for a whole iteration. Finally, in Section \ref{sec:convergence} we show how to extract a convergent subsequence and identify its limit as a weak solution to system \eqref{eq:complete}.

\subsection{Estimates for reaction step\label{sec:estim_reac}}
Since the right-hand sides, $\Sigma(\sigma,r), R(\sigma,r)$ are Lipschitz continuous in both components we note that the solution of the reaction system is unique. 
\begin{proposition}[$L^\infty$ estimates of the reaction step]
\label{prop:linfty_reaction}
Let $(\rn, \sigman)$ be given by our splitting scheme. Then there holds
\begin{align*}
	 0\leq\sigmahalf \leq \Linfty{\sigman} \exp(c\tau), \mbox{ and } 0\leq \rhalf \leq 1.
\end{align*}
\end{proposition}
\begin{proof}
We show first that there holds $\rhalf \in [0,1]$. Assume the contrary, \emph{i.e.} $\rhalf<0$ or $\rhalf>1$. If $\rhalf<0$ then, by continuity, there exists a time $t^\star\in(t^n,t^{n+1})$ such that $r(t^\star)=0$ and $\partial_t r(t^\star) < 0$. However, this is absurd as
\begin{align*}
	0 > \partial_t r(t^\star) = R\big(\sigma(t^\star), r(t^\star)\big) = \tilde G_1\big(\sigma(t^\star), 0\big) \geq 0.
\end{align*}
Analogously, it can be shown that $\rnp\leq1$. Finally, for the positivity of $\sigma$ we can use a similar argument. Let us assume $\sigman\geq0$ and $\sigmahalf<0$. Then, there exists another $t^\star$ such that
\begin{align*}
	0> \partial_t\sigma(t^\star) = \Sigma\big(\sigma(t^\star), r(t^\star)\big) = 0,
\end{align*}
which clearly is a contradiction. For the $L^\infty$-bound we simply apply Gronwall's lemma and the fact that $r\in[0,1]$.
\end{proof}

Next we address the BV-estimates during the reaction step. 


\begin{proposition}[Bounded variation of $\rhalf$ and $\sigmahalf$]
\label{prop:bv_reaction}
Let us consider $(\rn, \sigman)$ as initial data for our splitting scheme. Then,  the reaction step is BV-stable in the following sense.
	\begin{align*}
		\BV{\rhalf} + \BV{\sigmahalf} \leq (\BV{\rn} + \BV{\sigman})\,\exp(c \tau), 
	\end{align*}
    for some positive $K$, depending only on the Lipschitz constants of $F_i,G_i$ and the $L^\infty$-bounds on $F_i,G_i$, for $i=1,2$.
\end{proposition}
\begin{proof}
	Using the transformed system, Eqs.\eqref{eq:transf_sys1}, $r$ and $\sigma$ satisfy the following equations in the reaction step
    \begin{align*}
    	\partial_t \sigma = \Sigma(\sigma, r), \quad \mbox{ and }\quad\partial_t r = R(\sigma, r).
    \end{align*}
    Upon integrating in time we get
    \begin{align*}
    	\sigma(t) = \sigma(s) + \int_s^t \Sigma\big(\sigma(\bar\tau), r(\bar \tau)\big) \d \bar \tau, \quad \mbox{and} \quad     	r(t) = r(s) + \int_s^t R\big(\sigma(\bar\tau), r(\bar \tau)\big) \d \bar \tau.
    \end{align*}
    Now, let $P\subset \Omega$ be an arbitrary partition. We compute the variation of $\sigma$ and $r$ with respect to $P$ and obtain
    \begin{align*}
    	Q(t):= V_P(\sigma(t)) + V_P(r(t)) \leq V_P(\sigma(s)) + V_P(r(s)) + \int_s^t V_P(\Sigma(\bar \tau)) + V_P(R(\bar \tau))\d \bar \tau,
    \end{align*}
    whence
	\begin{align*}
		Q(t) \leq Q(s) + c\int_s^t Q(\bar \tau) \d \bar \tau,
	\end{align*}
	where $c$ only depends on the $L^\infty$-bounds and the Lipschitz-continuity of $A_i$, for $\in\{1,2,3\}$ and the $L^\infty$-bounds of $\sigma$ and $r$.
	Applying Gronwall's lemma we finally obtain
	\begin{align*}
		Q(t)\leq Q(s)\exp\big(c(t-s)\big).
	\end{align*}
	Passing to the supremum on the right-hand side and then on the left-hand side yields the result.

\end{proof}

\subsection{Estimates for diffusive step\label{sec:estim_diffu}}
This section is devoted to establishing $BV$-estimates and $L^\infty$-bounds in the diffusive step. To this end we will make use of the following remark.
\begin{remark}[Same optimality conditions]\label{rem:opt}
Let $\rhonp$ and $\etanp$ be given by the JKO step for $\mathcal{E}$, \emph{cf.} \eqref{eq:JKO}. Then, the optimality conditions, Eq. \eqref{eq:optimality_condition}, take the following form:
\begin{equation*}
	\frac{\delta \mathcal{E}}{\delta \rho} + \frac{\varphi_\rho}{\tau} = c_1, \quad \frac{\delta \mathcal{E}}{\delta \eta} + \frac{\varphi_\eta}{\tau} = c_2, \quad \text{and} \quad
  \mbox{with }  \frac{\delta \mathcal{E}}{\delta \rho} = \frac{\delta \mathcal{E}}{\delta \eta} = \chii(\sigma^{n+1}),
\end{equation*}
for some constants $c_1$ and $c_2$, and the respective Kantorovich potentials $\varphi_\rho$ (from $\rhonp$ to $\rhohalf$) and $\varphi_\eta$ (from $\etanp$ to $\etahalf$), see \cite{Snt_B}. Notice that the previous identities imply that $\varphi_\rho=\varphi_\eta=\varphi$ up to an additive constant. Moreover, the uniqueness of the optimal map implies that $\curlyT(x) = (\id - \partial_x \varphi)(x)$ is also the optimal map from $\sigmanp$ to $\sigmahalf$, and $\varphi$ is the corresponding Kantorovich potential.
\end{remark}

\begin{proposition}[$L^\infty$ stability of the diffusive step]
\label{prop:linfty_diffusion}
Let $(\rhalf, \sigmahalf)$ be given by the splitting scheme \eqref{eq:r_step}. Then these quantities satisfy
\begin{align*}
	0 \leq \sigmanp \leq \Linfty{\sigmahalf},\quad \text{ and } \quad 0 \leq \rnp \leq 1,
\end{align*}
after the diffusion step for any $1\leq n \leq N$.
\end{proposition}
\begin{proof}
Choose $x_0 \in \argmax (\sigmanp)$. Then, by the optimality condition, Eq. \eqref{eq:optimality_condition},

\begin{align*}
	\chi'(\sigmanp) + \frac{\varphi}{\tau} = c,
\end{align*}
we have $x_0 \in \argmin(\varphi)$ where we used the fact that $\chi''\geq 0$, \emph{cf.} Definition \ref{def:feasible_NLD}, $(N\!L-iii)$. Hence $\varphi''(x_0)\geq 0$ and consequently, by passing to the derivative in Eq. \eqref{eq:rel_T_KantorovichPotential} we get
\begin{align}
	\label{eq:T_prime_less_than_one}
	\curlyT'(x_0) = 1 - \partial_{xx} \varphi(x_0) \leq 1.
\end{align}
where $\curlyT$ is the transport map from $\sigmanp$ to $\sigmahalf$.  After a change of variables we get
\begin{align*}
	\sigmanp(x) &\leq \|\sigmanp\|_{L^\infty}= \sigmanp(x_0)=\curlyT'(x_0) \sigmahalf(\curlyT(x_0))\!\stackrel{\eqref{eq:T_prime_less_than_one}}{\leq} \sigmahalf(\curlyT(x_0))\\
    & \leq  \|\sigmahalf\|_{L^\infty},
\end{align*}
for any $x\in \Omega$. For the non-negativity we observe that $\curlyT'\geq 0$ by Brenier's theorem \cite{Bre87, Bre91, Vil_1, Vil_2}. Thus
\begin{align*}
	\sigmanp(x)= \sigmahalf\circ \curlyT(x) \curlyT'(x)\geq 0.
\end{align*}
Finally the bounds for $\rnp$ follow from its definition, \emph{cf.} Eq. \eqref{eq:JKOrsigma}, as the composition with a monotone function does not change  the infimum and the supremum of a function. This concludes the proof.
\end{proof}

\begin{proposition}[Bounded variation of $\rnp$ and $\sigmanp$]
\label{prop:bv_diffusion}
Let $(\rhalf, \sigmahalf)$ be given. After the diffusive step they satisfy the following estimate.
\begin{align*}
	\BV{\sigmanp} \leq \BV{\sigmahalf}, \qquad \text{and} \qquad \BV{\rnp} \leq \BV{\rhalf}.
\end{align*}
\end{proposition}
\begin{proof}
The result for the BV-norm of the minimiser, $\sigmanp$, is shown analogously to the proof of  Theorem 1.1., \emph{cf.} \cite{DPMSV16}.
Now we need to show that the BV-norm of the ratio $r$ does not increase. Recall the definition of $r^{n+1}$, \emph{cf.} Eq. \eqref{eq:JKOrsigma}, as 
\begin{align*}
	r^{n+1}:=r^{n+1/2} \circ \curlyT,
\end{align*}
where $\curlyT$ is the transport map such that $\rhohalf=\curlyT_\#\rhonp$ and $\sigmahalf = \curlyT_\#\sigmanp$. Note that it is indeed the same function and there holds $\curlyT'\geq 0$, by Brenier's theorem \cite{Bre87, Bre91, Vil_1, Vil_2} .

Now, let  $P\subset\Omega$  be any partition of $\Omega$. There holds
\begin{align*}
	V_P(r^{n+1}) = V_P(r^{n+1/2} \circ \curlyT) = V_{P'}(r^{n+1/2}) \leq \BV{r^{n+1/2}},
\end{align*}
where $P'$ is another partition induced by the monotone map $\curlyT$. Taking the supremum over all partitions $P$, we get 
\begin{align*}
	\BV{r^{n+1}} \leq \BV{r^{n+1/2}}.
\end{align*}

Finally note that, indeed,
\begin{align*}
	\rnp(x)  = \frac{\rhonp(x)}{\sigmanp(x)},
\end{align*}
on $\supp(\sigmanp)$ as we shall see now. According to Remark \ref{rem:opt}, the  same transport map $\curlyT$ pushes $\rhonp$ onto $\rhohalf$ and $\sigmanp$ onto $\sigmahalf$. As a consequence the densities satisfy
\begin{align*}
	\rhonp(x) = \rhohalf\big(\curlyT(x)\big) \, \curlyT'(x), \quad \text{and} \quad \sigmanp(x) = \sigmahalf\big(\curlyT(x)\big)\, \curlyT'(x),
\end{align*}
whence
\begin{align}
\label{eq:composition_with_monotone_T}
	\rnp(x) :=  \rhalf\circ\curlyT(x)=\frac{\rhohalf}{\sigmahalf}\circ\curlyT (x) =\frac{\rhonp}{\sigmanp}(x).
\end{align}
\end{proof}

\subsection{Combined estimates for an entire splitting step\label{sec:combi_estim}}
We have now garnered all information necessary to pass to the limit. Let us combine the estimates from the previous section in the following lemma.
\begin{lemma}[\label{lem:BVLinfty}$BV$-estimates and $L^\infty$-estimates]
The sequence $(r_\tau, \sigma_\tau)_{\tau>0}$ \ obtained by the splitting scheme is uniformly bounded in $BV(\Omega)$ and $L^\infty(\Omega)$. More precisely there holds
\begin{align*}
	\sup\limits_{t\in [0,T]} \Linfty{r_\tau} \leq C, \qquad \mbox{and} \qquad \sup\limits_{t\in[0,T]} \Linfty{\sigma_\tau} \leq C,
\end{align*}
and
\begin{align*}
	\sup\limits_{t\in [0,T]}  \BV{r_\tau}\leq C, \qquad\mbox{ and }\qquad \sup\limits_{t\in [0,T]} \BV{\sigma_\tau} \leq C,
\end{align*}
for some positive constant $C<\infty$ only depending on $T$.
\end{lemma}
\begin{proof}
	The uniform $L^\infty$-bound is a consequence of combining Propositions \ref{prop:linfty_reaction} and \ref{prop:linfty_diffusion}. We use these uniform $L^\infty$-bounds in the estimates for the $BV$-norm, \emph{cf.} Propositions \ref{prop:bv_reaction} and \ref{prop:bv_diffusion}. Combining both for the reaction and diffusion step we also obtain the uniform $BV$-bounds.
\end{proof}

As a result of Lemma \ref{lem:BVLinfty} and Lemma \ref{lem:Ralgebra} we obtain the following corollary.
\begin{corollary}[$BV$-estimates and $L^\infty$-estimates for $\rho, \eta$]
The sequences of approximated densities $(\rho_\tau)_{\tau>0}$, $ (\eta_\tau)_{\tau>0}$ are uniformly bounded in  $BV(\Omega)\cap L^\infty(\Omega)$. 
\end{corollary}

Finally, we need to prove an estimate on the cross-diffusion term to be able to pass to the limit later. This estimate is achieved in Lemma \ref{lem:uniformH1_bound} which is preceded by two technical lemmas. We exploit the existence of an auxiliary functional guaranteed by Definition \ref{def:feasible_NLD}. Note that in the absence of the reaction part this would indeed be an entropy in the classical sense, \emph{i.e.} that it is decayed along solutions. Since we are interested in a uniform estimate we shall begin by proving a control of this functional during each reaction phase.

\begin{lemma}[\label{lem:H1_part1}Control of the auxiliary functional in the reaction step]
	$\curlyK$ increases at most at a constant rate independent of $n$. More precisely there holds
	\begin{align*}
		\curlyK(\sigma^{n+1/2}) \leq \curlyK(\sigma^n)+c\tau,
	\end{align*}
	for any $n\in\N$.
\end{lemma}
\begin{proof}
	A straight forward computation yields
	\begin{align*}
		\ddt\int_\Omega K(\sigma)\d x &=\int_\Omega K(\sigma) \partial_t \sigma \d x = \int_\Omega K(\sigma) \sigma (A+ r B)\d x\leq c,
	\end{align*}
	using the uniform $L^\infty$-bound on $\sigma$ and the fact that $|\Omega|<\infty$.
	Hence, 
	$$
		\curlyK(\sigma^{n+1/2}) \leq \curlyK(\sigma^n)+c\tau,
	$$
	where $c$ is independent of $n$.
\end{proof}

Next, we address the diffusion step. As mentioned earlier the auxiliary functional, $\curlyK$, is an entropy for the diffusive part and from its dissipation we obtain the  necessary regularity, as asserted in the following lemma.
\begin{lemma}[$H^1$-bound for $\sigma^{n+1}$]
\label{lem:H1_part2}
The minimiser of the JKO step satisfies the following estimate
\begin{align*}
	\tau \|\partial_x \sigma^{n+1}\|_{L^2(\Omega)}^2 \leq  \left( \curlyK(u^{n+1/2}) - \curlyK(u^{n+1})\right),
\end{align*}
for each $1\leq n\leq N$.
\end{lemma}

\begin{proof}
Let 
$$
	(\rho^{n+1}, \eta^{n+1})\in \argmin \left\{\frac{1}{2\tau}\Wt(\cdot,U^{n+1/2}) + \frac12 \int_\Omega \eE(U)\d x \right\}. 
$$
Let $\sigma_s = (\curlyT_s)_\# \sigma^{n+1} $ be the geodesic interpolation between $\sigma_s|_{s=0} = \sigma^{n+1}$ and $\sigma_s|_{s=1} = \sigma^{n+1/2}$, given by
\begin{align*}
	\curlyT_s = (1-s)\id + s \curlyT,
\end{align*}
and
\begin{align*}
	\curlyT = \id - \partial_x \varphi,
\end{align*}
for the associated Kantorovich potential, $\varphi$, \emph{cf.} Eq. \eqref{eq:rel_T_KantorovichPotential}. As a consequence the velocity field is given by
\begin{align*}
	v_s = (\curlyT-\id) \circ \curlyT_s^{-1},
\end{align*}
satisfying the following continuity equation,
\begin{align*}
	\partial_s \rho_s = \partial_x(\rho_s v_s).
\end{align*}
We differentiate the entropy along the geodesic and  obtain 
\begin{align*}
	\label{eq:entropyevolution1}
	\frac{\d }{\d s}\int_\Omega K(\sigma_s) \d x &= \int_\Omega \kappa(\sigma_s) \partial_s \sigma_s \d x = - \int_\Omega \kappa'(\sigma_s)\partial_x\sigma_s \sigma_s v_s\d x\\
	&= - \int_\Omega \chi''(\sigma_s) \partial_x \sigma_s v_s\d x = - \int_\Omega\partial_x \chi'(\sigma_s) v_s \d x.
\end{align*}
Thus, at $s=0$, the evolution of the entropy Eq. \eqref{eq:entropyevolution1} becomes
\begin{align*}
	\ddt \int_\Omega K(\sigma_s) \d x\bigg|_{s=0} = \int_\Omega \partial_x \chi'(\sigma^{n+1})  \partial_x \varphi \d x.
\end{align*}
Using the optimality condition, Eq. \eqref{eq:optimality_condition}, we obtain
\begin{align*}
	 \tau \int_\Omega \big|\partial_x \chi'( \sigma^{n+1})\big|^2 \d x  &=  -\int_\Omega \partial_x \chi'(\sigma^{n+1})\partial_x \varphi \d x\\
     &=\ddt \int_\Omega K(\sigma_s) \d x\bigg|_{s=0}\\
     &\leq \curlyK(\sigma^{n+1/2}) - \curlyK(\sigma^{n+1}),
\end{align*}
where the last inequality is a consequence of the geodesic convexity of the entropy, \emph{cf.} Remark \ref{rem:K_geodesic_convex}. This concludes the proof.
\end{proof}

We combine the previous lemmas to obtain the desired estimate for a full iteration and finally for the piecewise constant interpolation, $\sigma_\tau$. 

\begin{lemma}[\label{lem:uniformH1_bound}Uniform $L^2((0,T)\times\Omega)$-bound for $\partial_x \sigma_\tau$]
There holds
\begin{align*}
	\|\partial_x \chi'(\sigma_\tau)\|_{L^2((0,T)\times\Omega)} \leq C,
\end{align*}
for some positive constant depending only on $T$.
\end{lemma}
\begin{proof}
	This statement is a consequence of combining Lemma \ref{lem:H1_part1} and Lemma \ref{lem:H1_part2} to get
	\begin{align*}
		\tau \|\partial_x \chi'(\sigma^{n+1})\|_{L^2(\Omega)}^2 &\leq \curlyK(\sigma^{n+1/2}) - \curlyK(\sigma^{n+1})\\
        &\leq c\tau  +\curlyK(\sigma^{n})-\curlyK(\sigma^{n+1}).
	\end{align*}
	Summing over $n=0\ldots N-1$ gives
	\begin{align*}
		\|\partial_x \chi'(\sigma_\tau)\|_{L^2((0,T)\times\Omega)}^2 \leq cT + \curlyK(\sigma^0) - \inf\limits_{\sigma}\curlyK(\sigma) \leq C,
	\end{align*}
     which yields the statement of the lemma. 
\end{proof}

\begin{lemma}[Total-square $2-$Wasserstein distance estimates]
\label{lem:total_square_estimate}
For every $n\in\{1,\ldots,N\}$ consider two consecutive steps for \eqref{eq:JKO},  $\Uhalf=(\rhohalf,\etahalf)$ and $\Unp=(\rhonp,\etanp)$, then there exists a constant $C$ such that
\begin{align*}
 \frac{1}{2\tau}\sum_{n=0}^{N}\Wt (\Uhalf,\Unp)\leq C.
\end{align*}
\end{lemma}
\begin{proof}
Using the minimising property of $\Unp$ we have
\[
 \frac{1}{2\tau}\Wt (\Uhalf,\Unp)\leq \eE(\Uhalf)-\eE(\Unp).
\]
Adding and subtracting $\eE(\Un)$ on the right-hand side, and considering that
\begin{align}
	\label{eq:intermediate_eqn1}
	&\eE(\Uhalf)-\eE(\Un)\\
	&=\int_\Omega \chi(\sigmahalf) \d x-\int_\Omega \chi(\sigman) \d x\nonumber\\
	& = \int_\Omega \chi\left(\sigman + \int_{t^n}^{t^{n+1}}\sigma(r A_1+ (1-r)A_2)\d s \right) \d x -\int_\Omega \chi(\sigman) \d x\nonumber\\
	&\leq \chi'(\xi) \int_{t^n}^{t^{n+1}} \sigma(r A_1+ (1-r)A_2) \d s,\nonumber
\end{align}
due to a Taylor expansion of $\chi$, where $\xi\in[\sigman, \sigmahalf]$. Using the $L^\infty$-bounds on $\sigma, r$ and $F_i,G_i$, we obtain
\begin{align*}
	&\eE(\Uhalf)-\eE(\Un)\leq C\tau,
\end{align*}
whence
\begin{equation*}
 \frac{1}{2\tau}\Wt (\Uhalf,\Unp)\leq \eE(\Un)-\eE(\Unp)+C\tau,
\end{equation*}
for a full time step. 
Finally, summing  over  $n$ gives  the result:
\begin{align*}
 \frac{1}{2\tau}\sum_{n=1}^{N}\Wt (\Uhalf,\Unp) & \leq \sum_{n=0}^{N}\left(\eE(\Un)-\eE(\Unp)+C\tau\right)\\
 &=\eE(U^0)-\eE(U_\tau^N)+N\left(C\tau\right)\\
 & \leq \eE(U^0)-\inf\limits_{U\in \mathcal{M}_+\times\mathcal{M}_+}\eE(U)+CN \tau\\
 &\leq \bar{C}.
\end{align*}
\end{proof}
\subsection{Convergence\label{sec:convergence}}
We now prove the convergence result.
\begin{proposition}\label{prop:ArzelaAscoli}
The piecewise constant interpolations defined in Definition \ref{def:piec_con} admit subsequences converging uniformly to absolutely continuous curves  $\bar{\rho},\bar{\eta}$ with  values in $\curlyM_+(\Omega)$. Moreover, $\bar\rho$ and $\bar\eta$ are $\Bl$-continuous functions on $[0,T]$.
\end{proposition}
\begin{proof}
The proof is based on the application of a generalised version of the Ascoli-Arzel\`a theorem, \emph{cf.} Ref. \cite{AGS}, Section 3.

We begin by establishing `almost continuity' of the approximation. To this end let $0\leq s<t\leq T$ be two time instances. Then there exist two uniquely determined integers, $m,k$, satisfying
\begin{align*}
	s\in((m-1)\tau,m\tau], \quad \mbox{and}\quad  t\in((k-1)\tau,k\tau],
\end{align*}
such that
\begin{align}
	\label{eq:holder_cty_estimate}
	\Bl(U_\tau(s),U_\tau(t)) &\leq \sum_{n=m}^{k-1}\Bl(\Un,\Unp)\leq\left(\sum_{n=m}^{k-1}\Bl^2(\Un,\Unp)\right)^{\frac12}|k-m|^{\frac12},
\end{align}
where $U_\tau(t)=\left(\rho_\tau(t),\eta_\tau(t)\right)$ as defined  in Definition \ref{def:piec_con}. It becomes apparent that we need to address the bounded Lipschitz term next. To this end we use the triangulation established in Corollary \ref{cor:triangular:_inequality} to estimate it by the $L^1$-distance in the reaction step and the $W_1$-distance in the diffusion step. Hence
\begin{align*}
	\sum_{n=m}^{k-1}\Bl^2(\Un, \Unp) \leq 2 \sum_{n=0}^{N}  \|\Un -  \Uhalf\|_{L^1(\Omega)}^2 + 2 \sum_{n=0}^{N} \Wo^2(\Uhalf,\Unp).
\end{align*}
For the reaction step an argument similar to Eq. \eqref{eq:intermediate_eqn1} yields
\begin{align*}
	\|\Un -  \Uhalf\|_{L^1(\Omega)}^2\leq C\tau^2,
\end{align*}
and, using that the $p$-Wasserstein distances are ordered, we even have 
\begin{align*}
	\sum_{n=m}^{k-1}\Bl^2(\Un, \Unp) &\leq 2 \sum_{n=0}^{N}\|\Un -  \Uhalf\|_{L^1(\Omega)}^2 + 2 \sum_{n=0}^{N}\Wt(\Uhalf,\Unp)\\
	&\leq C\tau,
    \end{align*}
    where we used the total-square estimate, Lemma \ref{lem:total_square_estimate}.
Therefore Eq. \eqref{eq:holder_cty_estimate} becomes
\begin{align*}
		d_{BL}(U_\tau(s),U_\tau(t))
		&\leq\left(\sum_{n=m}^{k-1}\Bl^2(\Un,\Unp)\right)^{\frac12}|k-m|^{\frac12}\\
		&\leq C\sqrt{\tau} \left(\frac{|t-s|}{\tau}+1\right)^{\frac12}\\
		&\leq C (\sqrt{|t-s|} + \sqrt{\tau}).
\end{align*}
Thus we get the `almost $\frac{1}{2}$-H\"{o}lder continuity' for the curve $U_\tau(t)$ and we obtain the uniform narrow compactness on compact time intervals by using the refined version of Ascoli-Arzel\`a's theorem, \emph{cf.} \cite{AGS}, Section 3.
\end{proof}

\begin{corollary}[\label{cor:strong_convergence}Strong convergence in $L^p(0,T;L^q(\Omega))$]
	Let $1 \leq p,q < \infty$ and  $(\rho_\tau)_{\tau>0}$ and $(\eta_\tau)_{\tau>0}$ be the sequences of the piecewise constant interpolations as in Definition \ref{def:piec_con}. Then there exist two functions $\rho,\eta\in L^p(0,T;L^q(\Omega))$ and subsequences, again denoted by $(\rho_\tau)_{\tau>0}$ and $(\eta_\tau)_{\tau>0}$, such that
	\begin{align*}
		\rho_\tau \rightarrow \rho, \qquad \mbox{ and } \qquad \eta_\tau \rightarrow \eta,
	\end{align*}
	strongly in $L^p(0,T;L^q(\Omega))$ as $\tau \rightarrow 0$. Moreover the convergence holds pointwise in time, \emph{i.e.} for all $t\in [0,T]$ there holds
	\begin{align*}
		\rho_\tau(t) \rightarrow \rho(t), \qquad \mbox{ and } \qquad \eta_\tau(t) \rightarrow \eta(t),
	\end{align*}
	strongly in $L^q(\Omega)$.	
\end{corollary}
\begin{proof}
	Note that it suffices to show the result for $(\rho_\tau)_{\tau>0}$ as the same argument applies for $(\eta_\tau)_{\tau>0}$. By Proposition \ref{prop:ArzelaAscoli} we can extract a subsequence, still denoted the same, such that
	\begin{align*}
		\rho_\tau(t)\rightharpoonup \rho,
	\end{align*}
	in $\curlyM_+(\Omega)$ for all $t\in [0,T]$. 
Furthermore, from the uniform $BV\cap L^\infty$-bounds, 
for any $t\in [0,T]$ the sequence converges strongly in $L^1(\Omega)$. Using the uniform $L^\infty$
-bounds and the dominated convergence theorem, we obtain the pointwise-in-time convergence in any $L^q(\Omega)$.
	
	Now let us apply the same argument on the whole domain, $[0,T]\times \Omega$. The pointwise convergence and the uniform $L^\infty(0,T;L^\infty(\Omega))$-bound imply
	\begin{align*}
		\rho_\tau \rightarrow \rho,
	\end{align*}
	strongly in $L^p(0,T;L^q(\Omega))$ by the dominated convergence theorem. This concludes the proof.
\end{proof}

\begin{lemma}[Identification of the limit]
The sequence constructed in \eqref{eq:r_step}, \eqref{eq:JKO} converges to a weak solution of Eqs. \eqref{eq:main}.
\end{lemma}
 \begin{proof} 
 The proof consists of two parts -- the diffusion part and the reaction part. We write them in their respective weak formulation and combine them to obtain the complete approximation of the weak formulation. Here we only show the argument for $\rho$ as the corresponding result for $\eta$ is shown analogously. We begin with the diffusion part. \\

\noindent
\textbf{Diffusion part.} 
We consider the two steps before and after the application of the JKO scheme \eqref{eq:JKO},
\begin{align*}
	\Uhalf=(\rhohalf,\etahalf)	, \qquad\mbox{and}\qquad\Unp=(\rhonp,\etanp).
\end{align*}
For a given test function $\zeta \in C_c^\infty(\Omega)$, let $\curlyT$ be the optimal transport map from  $\rhonp$ to  $\rhohalf$. Upon integration over $\Omega$ we get
\begin{equation*}
\frac{1}{\tau}\int_\Omega\left(\rhonp(x)-\rhohalf(x)\right)\zeta(x)\d x=\frac{1}{\tau}\int_\Omega\rhonp(x)\left(\zeta(x)-\zeta (\curlyT(x))\right)\d x.
\end{equation*}
Taylor expanding $\zeta(T(x))$ around $x$ yields
\begin{align*}
	\zeta(\curlyT(x)) &= \zeta(x + (\curlyT(x)-x))\\
		&=\zeta(x) + \partial_x\zeta(x)(\curlyT(x)-x) + \curlyO\left(|\curlyT(x)-x|^2\right).
\end{align*}
Moreover, using the fact that $\curlyT(x)=x-\partial_x \varphi(x)$, where $\varphi$ is the Kantorovich potential associated to the optimal map $\curlyT$, the integral above can be rewritten as
\begin{align*}
\frac{1}{\tau}\int_\Omega&\left(\rhonp(x)-\rhohalf(x)\right)\zeta(x)\d x\\
 & =-\frac{1}{\tau}\int_\Omega\rhonp(x)\left(\curlyT(x)-x\right)\partial_x\zeta(x)\d x+O\left(\Wt(\rhonp,\rhohalf)\right)\\
  & =\frac{1}{\tau}\int_\Omega\rhonp(x)\partial_x\varphi(x)\partial_x\zeta(x)\d x+O\left(\Wt(\rhonp,\rhohalf)\right).
\end{align*}
Thanks to the optimality condition, Remark \ref{rem:opt},
we get 
\begin{align}
	\begin{split}
	\label{eq:diffusion_step_identification}
\frac{1}{\tau}\int_\Omega&\left(\rhonp-\rhohalf\right)\zeta \d x\\
&=-\int_\Omega\rhonp\partial_x\chi'(\sigmanp)\,\partial_x\zeta \d x+\curlyO\left(\Wt(\rhonp,\rhohalf)\right).
	\end{split}
\end{align}

\noindent
\textbf{Reaction part.}
Note that 
\begin{align}
	\label{eq:reaction_step_identification}
	\begin{split}
	 \int_\Omega\frac{\rhohalf-\rhon}{\tau}\zeta \d x
	 &= \int_\Omega \frac{\rhalf\sigmahalf - \rn\sigman}{\tau}\zeta\d x\\
	 &=\int_\Omega\int_{t^n}^{t^{n+1}}\frac{1}{\tau} \partial_t \big(r(\bar\tau) \sigma(\bar \tau)\big)\d \bar \tau \d x,\\
	 &=\int_\Omega \left(\rhalf\sigmahalf \tilde F_1 + (1-\rhalf)\sigmahalf\tilde G_1\right) \d x + \curlyO(\tau)\\
	 &=\int_\Omega \left(\rhohalf F_1^{n+1/2}+\etahalf G_1^{n+1/2}\right) \zeta \d x+\curlyO(\tau),
	 \end{split}
\end{align}
with the shortcuts $F_1^{n+1/2}=F_1(\rhohalf,\etahalf)$ and $G_1^{n+1/2}=G_1(\rhohalf,\etahalf)$, having used Eq. \eqref{eq:rsigma_eqn}.

\noindent
\textbf{Combination of both steps.}
Let us combine the reaction step and the diffusion step of the splitting scheme. Upon summing up Eqs.(\ref{eq:diffusion_step_identification}, \ref{eq:reaction_step_identification})  we obtain
\begin{align}
\label{eq:almost_weak_form}
\begin{split}
 0 =& \frac{1}{\tau}\int_\Omega\zeta(x)(\rhonp(x)-\rhon(x))\d x+\curlyO(\Wt(\rhonp,\rhohalf))\\
 &+ \int_\Omega \rhonp\partial_x\chi'(\rhonp(x)+\etanp(x))\partial_{x} \zeta(x)\d x\\
 &  - \int_\Omega\left(\rhohalf F_1(\rhohalf,\etahalf)+\etahalf G_1(\rhohalf,\etahalf)\right) \zeta \d x\\
 &+\curlyO(\tau).
\end{split}
\end{align} 
We rewrite this equation in terms of the  piecewise constant interpolation, Definition \ref{def:piec_con}. For any $0<s<t <T$  there are two uniquely determined integers $m,k$ such that
\[
	s\in \left(m\tau,(m+1)\tau\right], \quad \mbox{and} \quad t\in\left(k\tau,(k+1)\tau\right].
\]
Multiplying Eq. \eqref{eq:almost_weak_form} by $\tau$ and summing from  $m$ to $k-1$, we obtain
\begin{align*}
 0 =& \int_\Omega\zeta(x)(\bar{\rho}_\tau(t,x)-\bar{\rho}_\tau(s,x))\d x+\curlyO(\tau)\\
 &+ \int_s^t\int_\Omega \bar{\rho}_\tau(\bar\tau,x)\partial_x\chi'\big(\bar{\rho}_\tau(\bar\tau,x)+\bar{\eta}_\tau(\bar\tau,x)\big)\partial_{x} \zeta(x)\\
 &  - \bigg(\bar{\rho}_\tau(\bar\tau,x) F_1\big(\bar{\rho}_\tau(\bar\tau,x),\bar{\eta}_\tau(\bar\tau,x)\big)+\bar{\eta}_\tau(\bar\tau,x) G_1\big(\bar{\rho}_\tau(\bar\tau,x),\bar{\eta}_\tau(\bar\tau,x)\big)\bigg) \zeta(x) \d x d\bar\tau.
\end{align*} 
We are now ready to pass to the limit. The strong convergence obtained in Corollary \ref{cor:strong_convergence} allows us to pass to the limit in the nonlinear reaction terms. Moreover the cross-diffusion term converges due to the weak-strong $L^2(0,T;L^2(\Omega))$-duality, by Lemma \ref{cor:strong_convergence} and Corollary \ref{cor:strong_convergence}. Passing  to the limit $\tau\to 0$ we obtain  the weak formulation for $\rho$. The same argument for $\eta$ yields the statement.
\end{proof}

We end the paper with a stunning result which can be seen as a generalisation of the result of Bertsch \emph{et al.}, \emph{cf.} \cite{BGHP85,BGH87}. In their papers they prove that initially segregated species stay segregated at all times. We can drop their assumption that 
\begin{align*}
	\supp(\rho_0) < \supp(\eta_0),
\end{align*}
\emph{i.e.} that both species are ordered and prove the following, more general theorem.
\begin{theorem}[\label{thm:segregation}Segregation in the case of no cross-reaction.]
	Assume no cross-reaction terms, \emph{i.e.} $G_1\equiv G_2\equiv 0$ and that both species are initially segregated, \emph{i.e.}
	\begin{align*}
		\int_\Omega \rho_0(x)\,\eta_0(x)\d x = 0.
	\end{align*}
	Then, there exists a solution such that both species stay segregated at all times.
\end{theorem}
\begin{proof}
	It suffices to show this property at the level of the discrete scheme since, once it is established, we use the strong $L^2$-convergence of $\rho_\tau$ and $\eta_\tau$ to show segregation is also kept in the limit:
	\begin{align*}
		0 = \lim_{\tau \to 0} \int_\Omega \rho_\tau(t,x)\eta_\tau(t,x) \d x = \int_\Omega \rho(t,x) \eta(t,x)\d x.
	\end{align*}
	Thus let us now show the property at the level of the approximation. It is clear that in the reaction step segregation is kept as it does not change the support of both $\rho$ and $\eta$. Hence we only need to prove that segregation is kept in the diffusion step. This is done by contradiction. Let us assume there exists an $1\leq n \leq N$ such that
	\begin{align}
		\label{eq:no_overlap}
		\int_\Omega \rhohalf(x) \etahalf(x)\d x = 0,
	\end{align}
	while
	\begin{align*}
		\int_\Omega \rhonp(x) \etanp(x) \d x > 0.
	\end{align*}
	Then there exists $\delta > 0$ and a set $B$ with $|B|>0$, such that
	\begin{align*}
		 |\curlyT'(x)|<\infty\,, \qquad \rhonp(x)> \delta\,, \quad\mbox{and}\quad \etanp(x) > \delta, 
	\end{align*}
	for almost all $x\in B$. As both species have the same transport map $\curlyT$ in common there exists a set $A$ such that $A = \curlyT(B)$ and
	\begin{align*}
		\begin{split}
			0 &< \delta < \rhonp(x) =   \rhohalf(\curlyT(x))\curlyT'(x),\\
			0 &< \delta <  \etanp(x)=  \etahalf(\curlyT(x))\curlyT'(x),
		\end{split}
	\end{align*}
	which is absurd, for we assumed \eqref{eq:no_overlap}. Thus segregation is kept at each iteration which concludes the proof of the theorem.	
\end{proof}

\section{\label{sec:numerics}Numerical simulations}
In this section we illustrate our main result by showing some numerical simulations using a different splitting scheme. This splitting scheme is based on the finite volume scheme introduced in \cite{BF,CCH15,CHS17}. We solve the cross-diffusion inner time step with this finite volume scheme instead of the optimal mass transport approach. Then we solve the reaction step along the ODEs. We refer to \cite{CKL17} for the numerical difficulties related to the implementation of the mass transport approach in a splitting fashion. 

\begin{figure}[ht!]
	\centering
	\subfigure[Initial data.]{
		\includegraphics[width=0.45\textwidth]{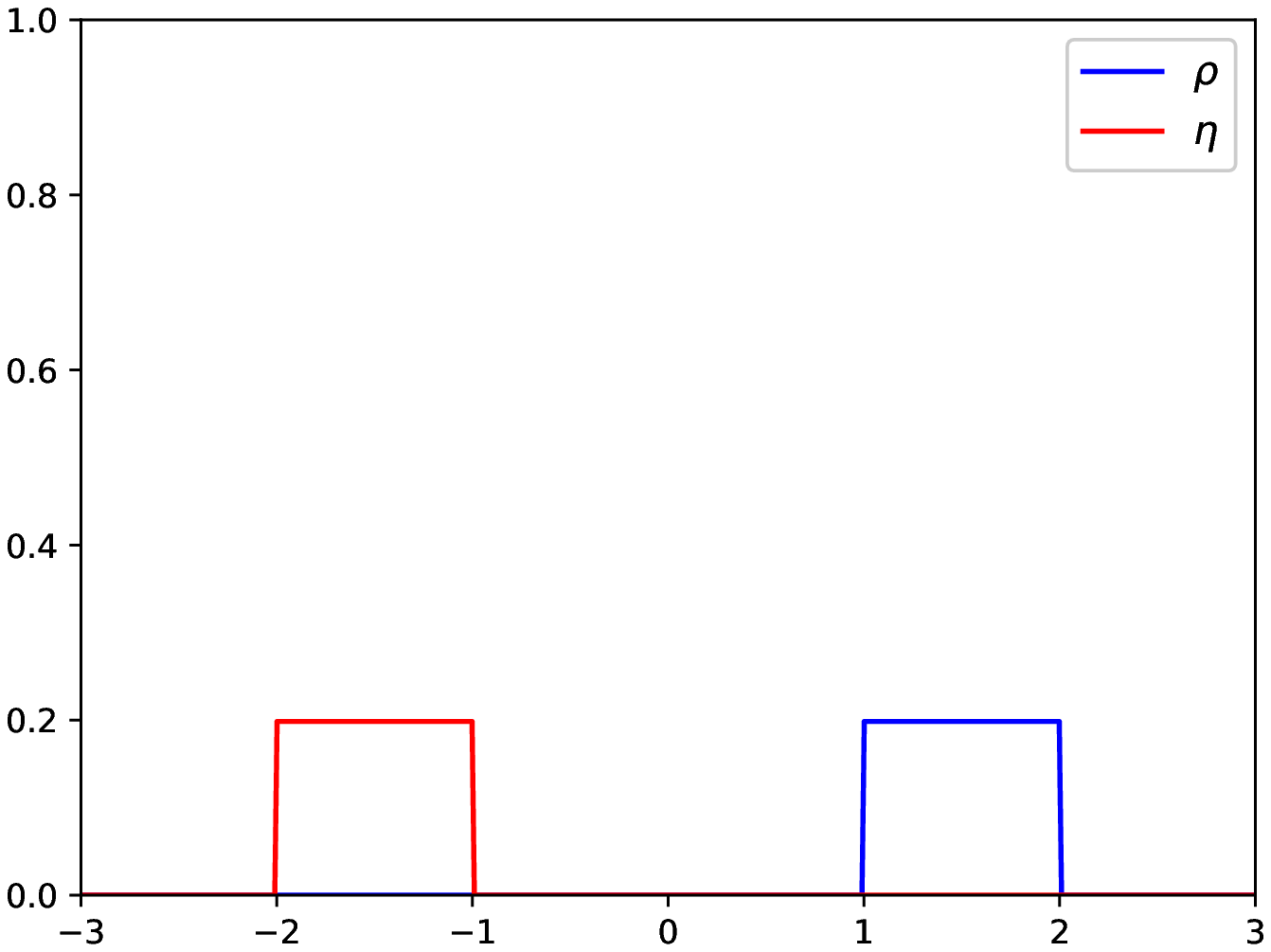}
	}
	\subfigure[Final time.]{
		\includegraphics[width=0.45\textwidth]{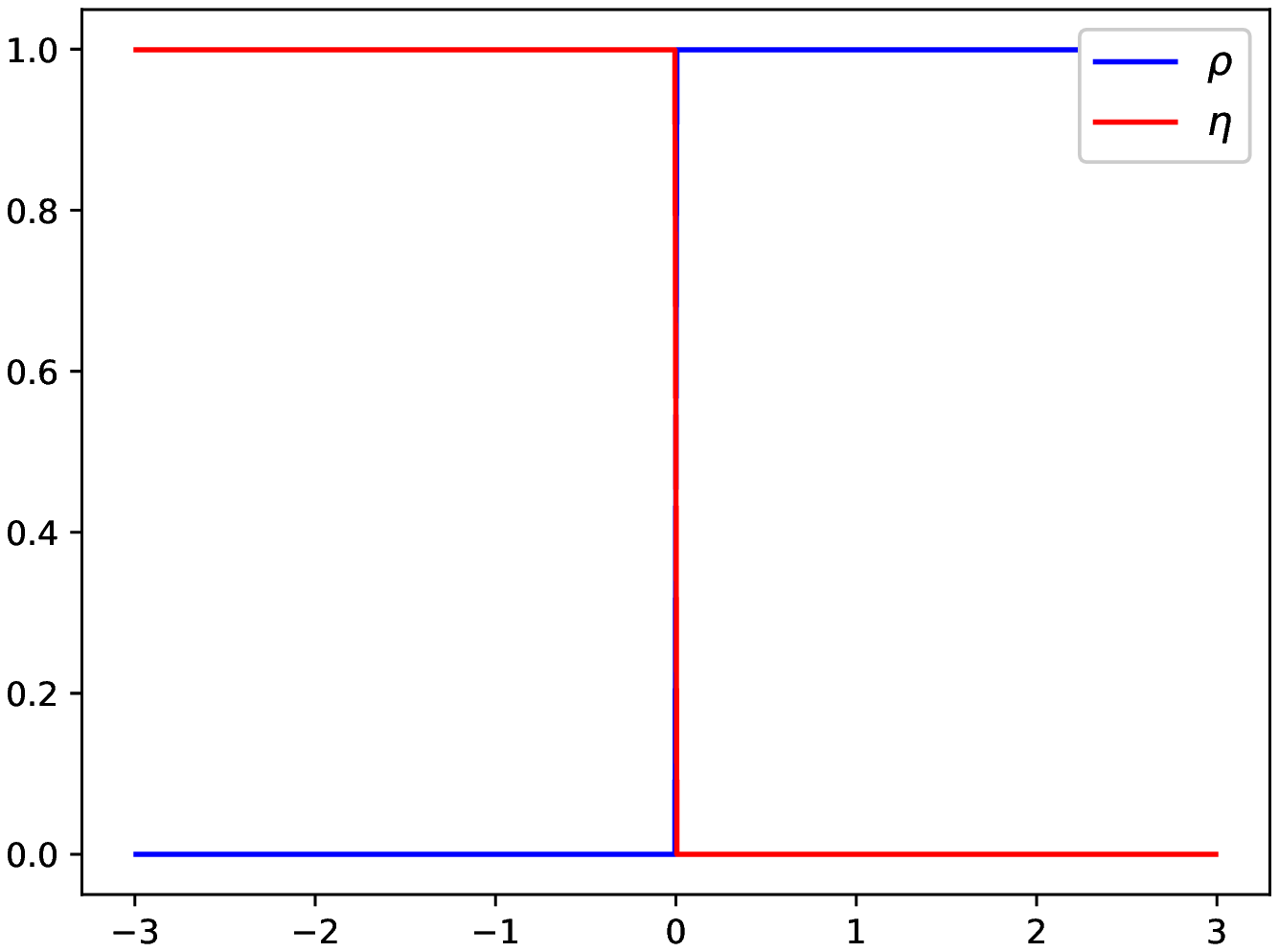}
	}
	\caption{System \eqref{eq:lotka_volterra} initialised with two indicator functions. Both species diffuse individually while increasing according to the reaction term. They remain segregated in agreement with Theorem \ref{thm:segregation} and reach a constant stationary state after some time.}
	\label{fig:lotka_volterra_init_n_final}
\end{figure}

\subsection{Lotka-Volterra type reaction}
Let us consider the system
\begin{align}
	\label{eq:lotka_volterra}
	\begin{split}
	\partial_t \rho &= \partial_x (\rho \partial_x(\sigma)) + \rho(1-\sigma),\\
	\partial_t \eta &= \partial_x (\eta \partial_x(\sigma)) + \eta(1-\sigma),
	\end{split}
\end{align}
modelling two species avoiding overcrowding due to the nonlinear cross-diffusion term as well. In addition the growth term is of Lotka-Volterra type, \emph{cf.} Figures \ref{fig:lotka_volterra_init_n_final} \& \ref{fig:lotka_volterra_intermediate}.

\begin{figure}[ht!]
 	\centering
 		\includegraphics[width=0.45\textwidth]{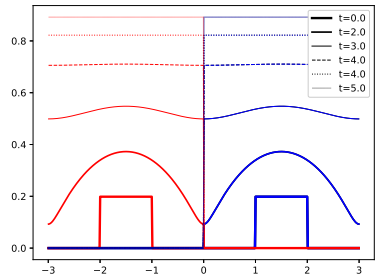}
 	\caption{Evolution in time of system \eqref{eq:lotka_volterra}.}
	\label{fig:lotka_volterra_intermediate}       
\end{figure}

Let us consider now a little modification of system \eqref{eq:lotka_volterra}, as in \cite{BerDaPMim},
\begin{align}
	\label{eq:lotka_volterra_BerDaPMim}
	\begin{split}
	\partial_t \rho &= \partial_x (\rho \partial_x(\sigma)) + \rho(1-\sigma),\\
	\partial_t \eta &= \partial_x (\eta \partial_x(\sigma)) + \eta(1-\sigma/2).
	\end{split}
\end{align}
Using the same parameters and domain as in \cite{BerDaPMim} Figures \ref{fig:lotka_volterra_BerDaPMim_init_n_final} \& \ref{fig:lotka_volterra_BerDaPMim_intermediate} show complete segregation in agreement with their result and moreover numerically validate our main result in Theorem \ref{thm:segregation}. 
 \begin{figure}[ht!]
 	\centering
 	\subfigure[Initial data.]{
 		\includegraphics[width=0.45\textwidth]{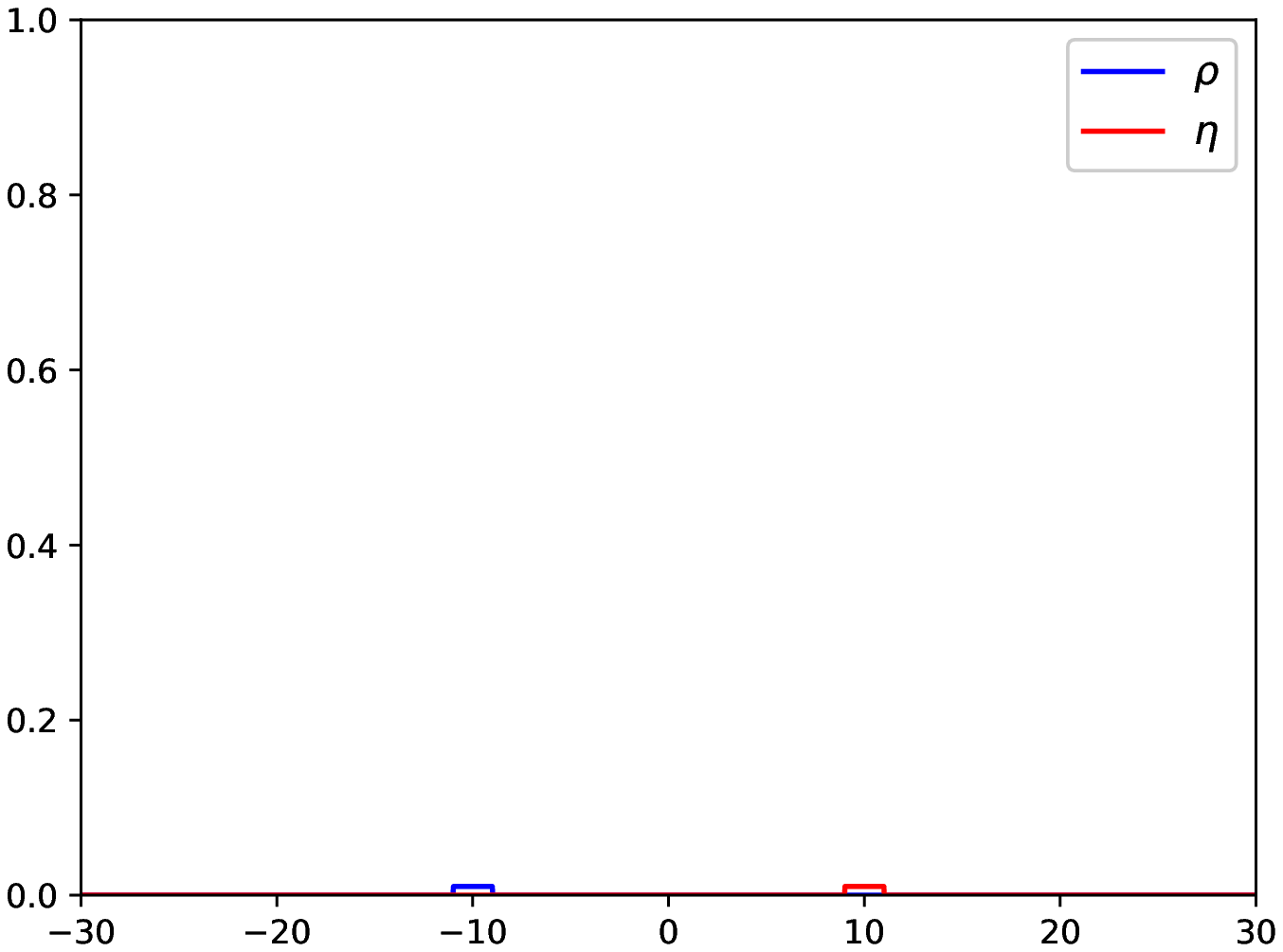}
 	}
 	\subfigure[Final time.]{
 		\includegraphics[width=0.45\textwidth]{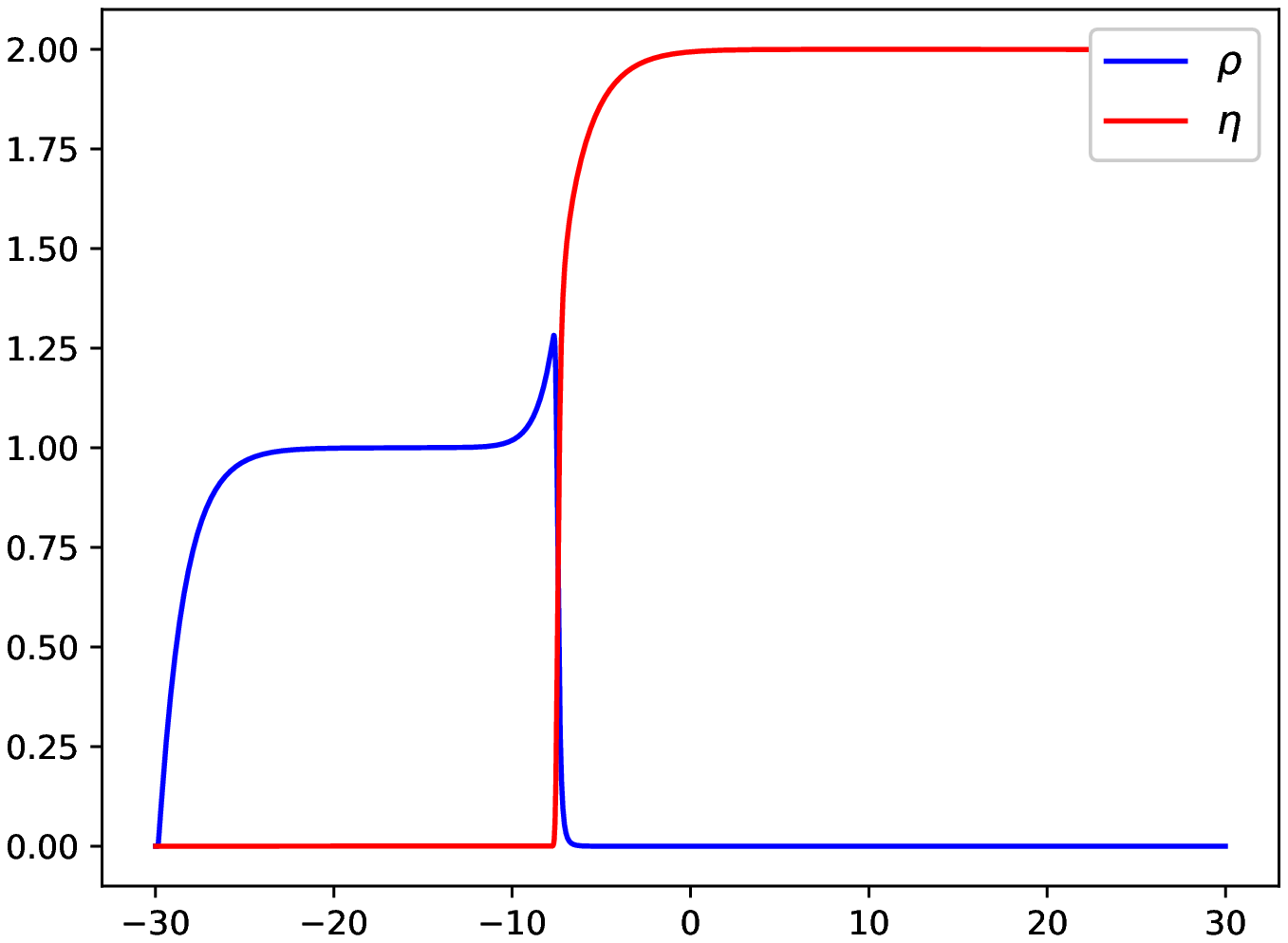}
 	}
 	\caption{System \eqref{eq:lotka_volterra_BerDaPMim} initialised with two indicator functions. Immediately both species start to grow due to the reaction terms while diffusing independently. In addition they remain segregated as proven by Theorem \ref{thm:segregation}.}
\label{fig:lotka_volterra_BerDaPMim_init_n_final}
\end{figure}

\begin{figure}[ht!]
  	\centering
  		\includegraphics[width=0.5\textwidth]{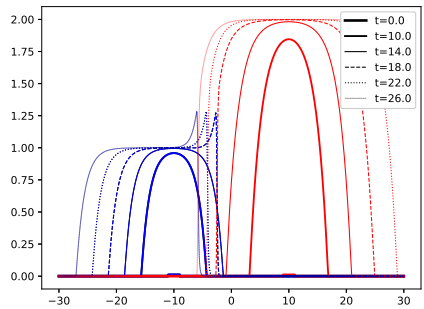}
  	\caption{Evolution in time of system \eqref{eq:lotka_volterra_BerDaPMim}.}
 	\label{fig:lotka_volterra_BerDaPMim_intermediate}       
 \end{figure}

In the second case, Figures \ref{fig:lotka_volterra_BerDaPMim_two_init_n_final} \& \ref{fig:lotka_volterra_BerDaPMim_two_intermediate}, we see the competition between both species clearly as one invades the other.   
 
\begin{figure}[ht!]
 	\centering
 	\subfigure[Initial data.]{
 		\includegraphics[width=0.45\textwidth]{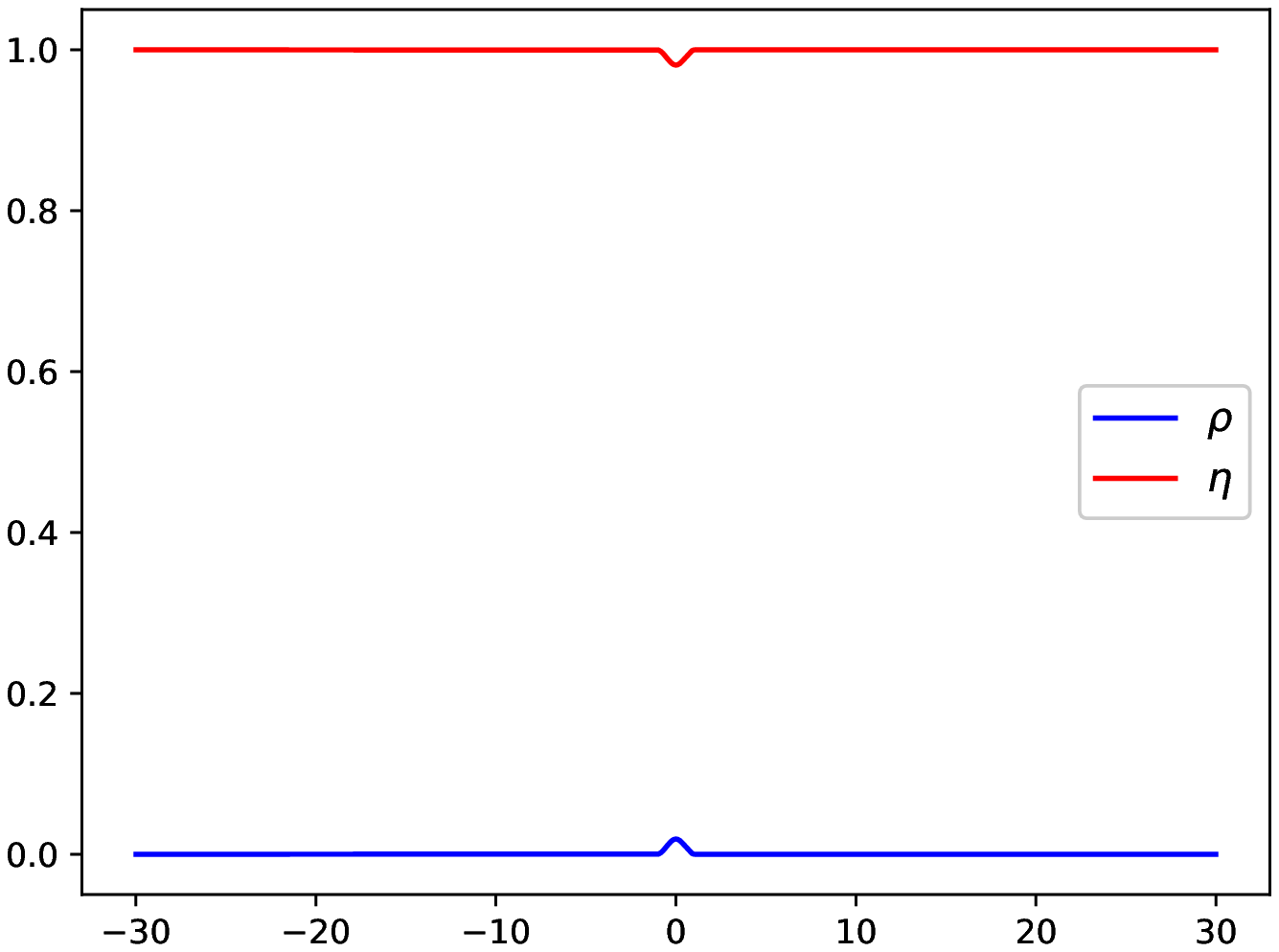}
 	}
 	\subfigure[Final time.]{
 		\includegraphics[width=0.45\textwidth]{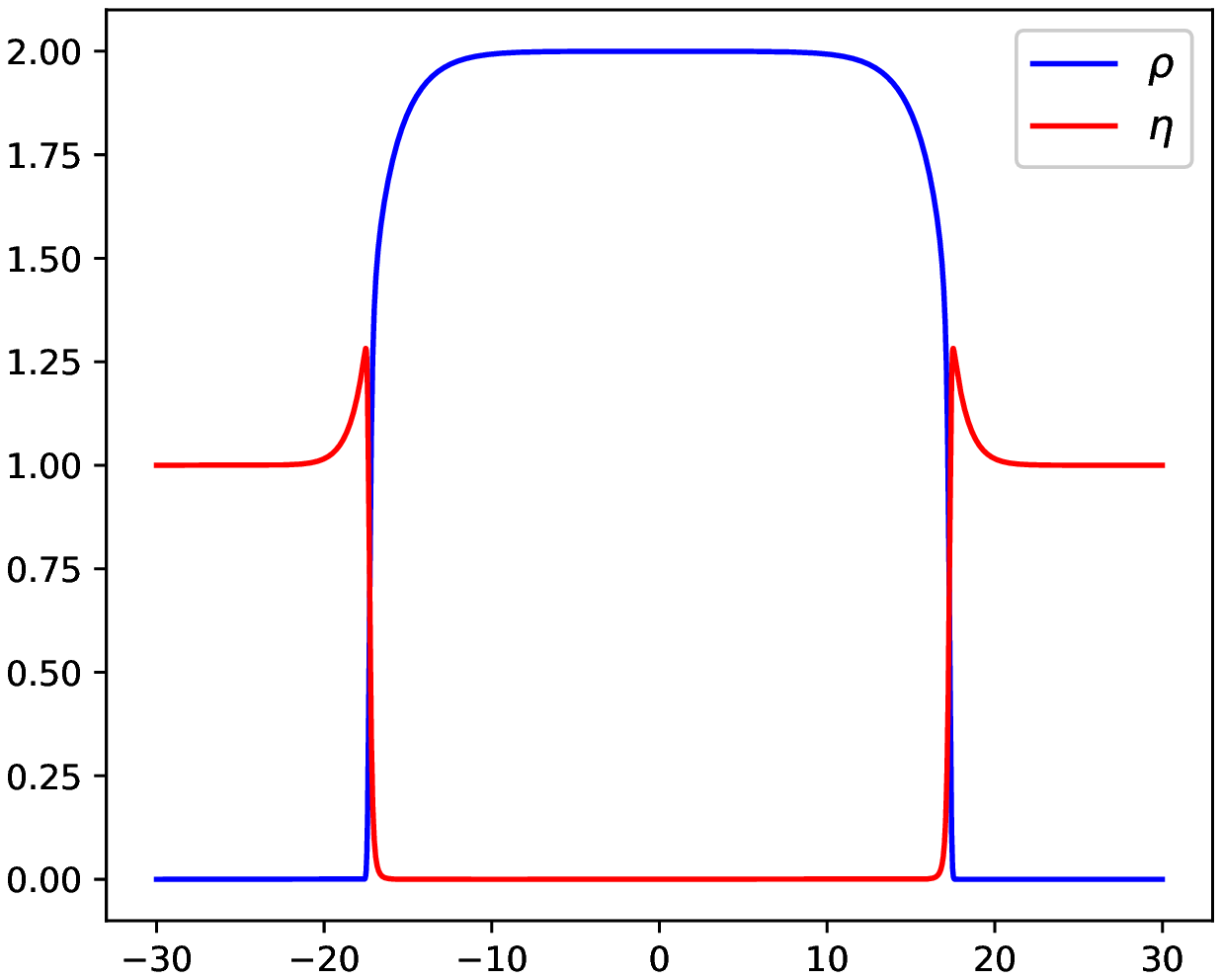}
 	}
 	\caption{System \eqref{eq:lotka_volterra_BerDaPMim} is initialised with $\rho(x)=\mathds{1}_{[-\pi/3,\pi/3]} (1+\cos(3x))$ and $\eta(x)= 1-\rho(x)$. Although both species intermingle initially the first species, $\rho$, grows rapidly and pushes aside the second species, $\eta$. After some time it seems that $\rho$ has completely eliminated $\eta$ on its support.}
 \label{fig:lotka_volterra_BerDaPMim_two_init_n_final}
 \end{figure}

\begin{figure}[ht!]
  	\centering
  		\includegraphics[width=0.5\textwidth]{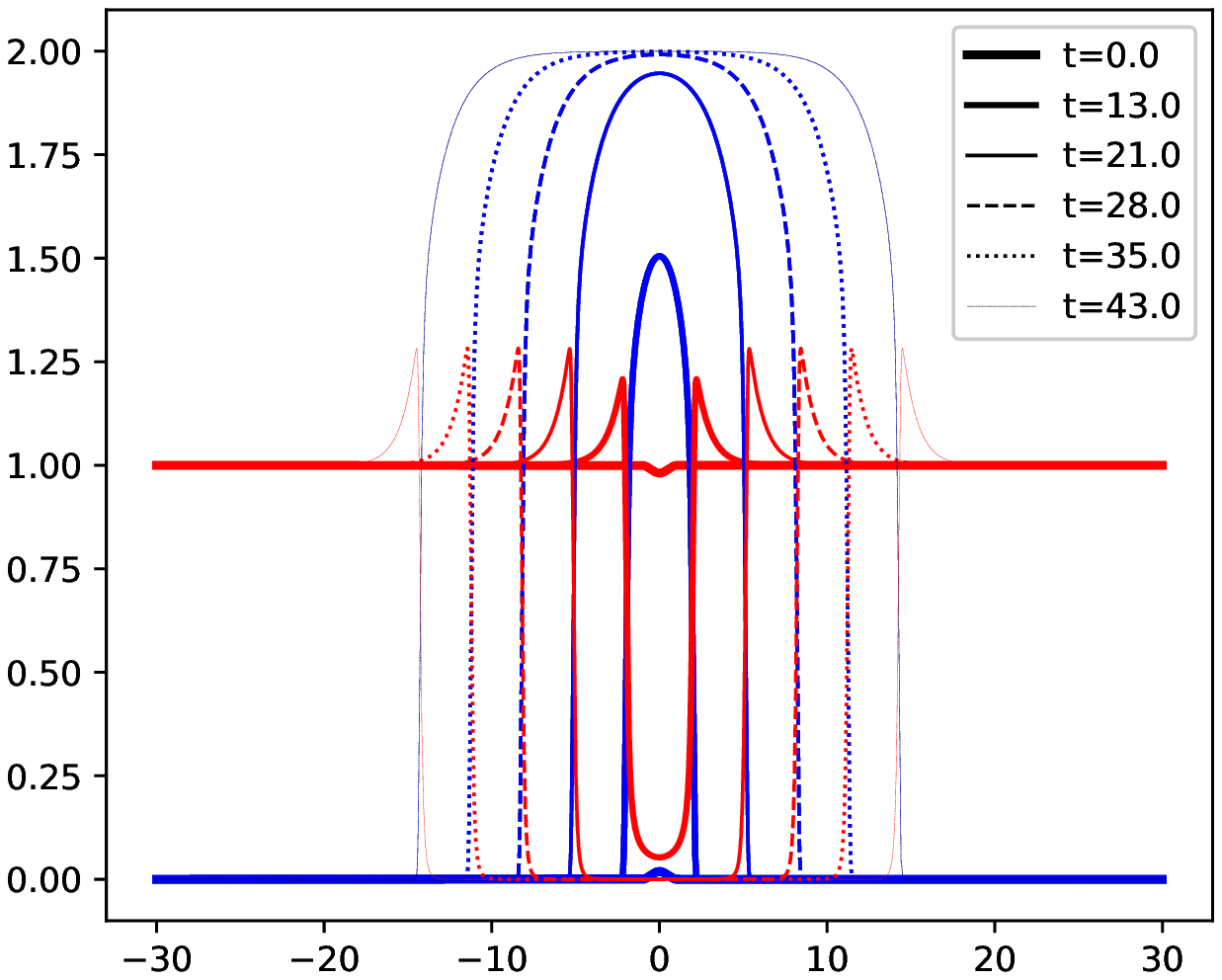}
  	\caption{Evolution in time of system \eqref{eq:lotka_volterra_BerDaPMim}.}
 	\label{fig:lotka_volterra_BerDaPMim_two_intermediate}       
 \end{figure}

\end{document}